\documentclass{amsart}
\usepackage{amsthm, amssymb, array, enumerate, times, txfonts, mathrsfs, bbm, verbatim}
\usepackage[centredisplay, PostScript=dvips]{diagrams}
\usepackage[initials]{amsrefs}
\pdfoutput=1
\title{Cubic algebras and Implication Algebras}
\author{Colin G.Bailey}
\address{School of Mathematics, Statistics \& Operations Research\\
Victoria University of Wellington\\
Wellington\\
NEW ZEALAND
}
\email{Colin.Bailey@vuw.ac.nz}
\author{Joseph S.Oliveira}
\address{
Pacific Northwest National Laboratories\\
Richland\\
U.S.A.}
\email{Joseph.Oliveira@pnl.gov}
\date{2009, February 4}

\let\rsf\mathscr

\let\bbm\mathbbm

\def\one{{\mathbf 1}}

\def\caret{\mathbin{\hat{\hphantom{m}}}}
\def\env{\operatorname{env}}
\def\rng{\operatorname{rng}}

\def\caret{\mathbin{\hat{\hphantom{i}}}}
\def\eqcl[#1]{\pmb{[}#1\pmb{]}}
\def\leftGen{{[\kern-1.1pt[}}
\def\rightGen{{]\kern-1.1pt]}}
\providecommand{\meet}{\mathbin{\wedge}}
\providecommand{\join}{\mathbin{\vee}}
\newcommand{\comp}[1]{\overline{#1}}
\ifx\upharpoonright\undefined
     \def\restrict{\hbox{\rm\kern0.166em\accent"12\kern-0.536em$\vert$\kern0.3em}}%
  \else
     \def\restrict{\upharpoonright}%
  \fi
\makeatletter
\def\twoSet#1#2{\left\{%
\vphantom{#2}#1\thinspace\right|\nolinebreak[3]\left.%
  #2%
  \vphantom{#1}%
  \right\}%
}
\def\oneSet#1{\left\lbrace#1\right\rbrace}

\newif\if@nstr
\def\setstrfalse{\let\if@nstr=\iffalse}
\def\setstrtrue{\let\if@nstr=\iftrue}
\def\@nstr #1#2{
\def\@@nstr ##1#1##2##3\@@nstr{\ifx
\@nstr ##2\setstrfalse \else \setstrtrue \fi }
\@@nstr #2#1\@nstr \@@nstr}
\def\@separate#1|#2@{\setFront{#1}\setBack{#2}}
\def\lb#1\rb{\@nstr|{#1} \if@nstr \@separate#1 @ \twoSet{\@setFront}{\@setBack}%
\else \@separate |{#1 }@ \oneSet{\@setBack}\fi%
}
\def\setFront#1{\def\@setFront{#1}}
\def\setBack#1{\def\@setBack{#1}}
\def\Set#1{\lb{#1}\rb}

\def\oneBrk#1{\left\langle#1\right\rangle}
\def\twoBrk#1#2{\left\langle%
\vphantom{#2}#1\thinspace\right|\nolinebreak[3]\left.%
  #2%
  \vphantom{#1}%
  \right\rangle%
}
\def\brk<#1>{\@nstr|{#1} \if@nstr \@separate#1 @ \twoBrk{\@setFront}{\@setBack}%
\else \@separate |{#1 }@ \oneBrk{\@setBack}\fi%
}

\def\thmref#1{\normalfont{theorem}~\ref{#1}}
\def\lemref#1{\normalfont{lemma}~\ref{#1}}

\def\corref#1{\normalfont{corollary}~\ref{#1}}

\theoremstyle{plain}
\newtheorem{thm}{Theorem}[section]
\newtheorem{lem}[thm]{Lemma}
\newtheorem{cor}[thm]{Corollary}

\newtheorem{defn}[thm]{Definition}

\theoremstyle{remark}
{}
{\newtheorem{example}{Example}[section]}
{}
{\newtheorem{rem}{Remark}[section]}

\@ifpackageloaded{bbm}{

}{

}
\makeatother

\begin{document}
\begin{abstract}
	We consider relationships between cubic algebras and implication 
	algebras.  
	We first exhibit  a functorial construction of a cubic algebra from an implication 
	algebra. Then we consider an collapse of a cubic algebra to an 
	implication algebra and the connection between these two operations. 
	Finally we use the ideas of the collapse to obtain a 
	Stone-type
	representation theorem for a large class of cubic algebras. 
\end{abstract}
\maketitle
\section{Introduction}
\subsection{Cubic Algebras}
Cubic algebras first arose in the study of face lattices of $n$-cubes 
(see\cite{MR:cubes}) and in 
considering the poset of closed intervals of Boolean algebras (see 
\cite{BO:eq}). Both of these families of posets have a 
partial binary operation $\Delta$ -- a generalized reflection. Cubic algebras then arise in full 
generality by taking the variety generated by either of these classes 
with $\Delta$, join and one. 

In this paper we consider another construction of cubic algebras from 
implication algebras. This construction produces (up to isomorphism) 
every countable cubic algebra. Cubic algebras also admit a natural collapse 
to an implication algebra. We show that this collapse operation is a one-sided inverse to  
this construction. 

A consequence of the Stone representation theorem for Boolean 
algebras is that the set of filters of a Boolean algebra is a Heyting    
algebra into which the original Boolean algebra embeds naturally. 
The collapsing process  for cubic algebras highlights certain filter-like subimplication algebras of 
cubic algebras that generate the algebra and let us do a similar 
construction for cubic algebras. Thus,  by looking at the set of 
all these subobjects we produce a new algebraic structure from which 
we can pick a subalgebra that is an MR-algebra. And our original 
cubic algebra embeds into it in a natural way. 

Before beginning our study we recall some of the basics of cubic and 
MR algebras. 

\begin{defn}
    A \emph{cubic algebra} is a join semi-lattice with one and a binary 
    operation $\Delta$ satisfying the following axioms:
    \begin{enumerate}[a.]
        \item  if $x\le y$ then $\Delta(y, x)\join x = y$;
        
        \item  if $x\le y\le z$ then $\Delta(z, \Delta(y, x))=\Delta(\Delta(z, 
        y), \Delta(z, x))$;
        
        \item  if $x\le y$ then $\Delta(y, \Delta(y, x))=x$;
        
        \item  if $x\le y\le z$ then $\Delta(z, x)\le \Delta(z, y)$;
        
        \item[] Let $xy=\Delta(1, \Delta(x\join y, y))\join y$ for any $x$, $y$ 
        in $\mathcal L$. Then:
        
        \item  $(xy)y=x\join y$;
        
        \item  $x(yz)=y(xz)$;
    \end{enumerate}
\end{defn} 

$\mathcal L$ together with $\brk<x, y>\mapsto xy$ is an implication 
algebra. More details on these algebras and some basic representation 
theory can be found in \cite{BO:eq}. A good reference for implication 
algebras is \cite{Abb:bk}.

\subsection{MR-algebras}

\begin{defn}
    An \emph{MR-algebra} is a cubic algebra satisfying the MR-axiom:\\
    if $a, b<x$ then 
    \begin{gather*}
        \Delta(x, a)\join b<x\text{ iff }a\meet b\text{ does not exist.}
    \end{gather*}
\end{defn}

\begin{example}
    Let $X$ be any set,  and 
    $$
    \rsf S(X)=\Set{\brk<A, B> | A, B\subseteq X\text{ and }A\cap 
    B=\emptyset}.
    $$
    Elements of $\rsf S(X)$ are called \emph{signed subsets} of $X$.
    The operations are defined by 
    \begin{align*}
	1&=\brk<\emptyset,  \emptyset>\\
	\brk<A, B>\join\brk<C, D>&=\brk<A\cap C,  B\cap D>\\
	\Delta(\brk<A, B>, \brk<C, D>)&=\brk<A\cup D\setminus B, 
	B\cup C\setminus A>.
\end{align*}
These are all atomic MR-algebras. The face-poset of an $n$-cube is 
naturally isomorphic to a signed set algebra. 
\end{example}

\begin{example}
    Let $B$ be a Boolean algebra, then the \emph{interval algebra} of $B$ is
$$
\rsf I(B)=\Set{[a, b] | a\le b \text{ in }B}
$$
ordered by inclusion. The operations are defined by
\begin{align*}
	1&=[0, 1]\\
	[a, b]\join[c, d]&=[a\meet c, b\join d]\\
	\Delta([a, b], [c, d])&=[a\join(b\meet\comp d), b\meet(a\join\comp c)].
\end{align*}
These are all atomic MR-algebras. For further details see \cite{BO:eq}.

We note that $\rsf S(X)$ is isomorphic to $\rsf I(\wp(X))$. 
\end{example}

\begin{defn}\label{def:caret}
    Let $\mathcal L$ be a cubic algebra. Then for any $x, y\in\mathcal L$ 
    we define the (partial) operation $\caret$ (\emph{caret}) by:
    $$
        x\caret y=x\meet\Delta(x\join y, y)
    $$
    whenever this meet exists. 
\end{defn}

The operation $\caret$ is used as a partial substitute for meets as 
the next lemma suggests. 

\begin{lem}
	If $\mathcal L$ is a cubic algebra and $x, y\in\mathcal L$ then --
	\begin{enumerate}[(a)]
		\item  
    $\mathcal L$ is an MR-algebra iff the caret operation is total.
	
		\item if $x\meet y$ exists then $x\meet y=x\caret\Delta(x\join y, y)$. 
	\end{enumerate}    
\end{lem}
\begin{proof}
    See \cite{BO:UniMR} lemma 2.4 and theorem 2.6.
\end{proof}

%

As in any algebra we have subalgebras. If $\mathcal L$ is a cubic 
algebra we denote by $\leftGen X\rightGen$ the subalgebra generated by $X$.  

\subsection{Enveloping Algebras}
We recall from \cite{BO:cong} the existence of \emph{enveloping 
algebras}. 

\begin{thm}[Enveloping Algebra]\label{thm:envAlg}
	Let $\mathcal L$ be any cubic algebra. Then there is an MR-algebra 
	$\env(\mathcal L)$ and an embedding $e\colon\mathcal 
	L\to\env(\mathcal L)$ such that:
	\begin{enumerate}[(a)]
		\item the range of $e$ generates $\env(\mathcal L)$; 
	
		\item the range of $e$ is an upwards-closed subalgebra; 
	
		\item any cubic homomorphism $f$ from $\mathcal L$ into an MR-algebra 
		$\mathcal N$ lifts uniquely to a cubic homomorphism $\widehat f$ from 
		$\env(\mathcal L)$ to $\mathcal N$.  Furthermore if $f$ is onto 
		or one-one then so is $\widehat f$.
	\end{enumerate}
\end{thm}

\begin{defn}\label{def:envAlg}
	Let $\mathcal L$ be any cubic algebra. Then the MR-algebra 
	$\env(\mathcal L)$ defined above is called the \emph{enveloping 
	algebra} of $\mathcal L$. 
\end{defn}

\section{Implication Collapse}
\begin{defn}
    Let $\mathcal L$ be a cubic algebra and $a, b\in\mathcal L$. Then
    \begin{align*}
        a\preceq b &\text{ iff }\Delta(a\join b, a)\le b\\
        a\sim b &\text{ iff }\Delta(a\join b, a)=b.
    \end{align*}
\end{defn}

\begin{lem}
    Let $\mathcal L$, $a$, $b$ be as in the definition. Then
    $$
    a\preceq b\text{ iff }b=(b\join a)\meet(b\join\Delta(1, a)).
    $$
\end{lem}
\begin{proof}
    See \cite{BO:eq} lemmas 2.7 and 2.12.
\end{proof}

\begin{lem}
	Let $\mathcal L$ be a cubic algebra and $a\in\mathcal L$. If $b, 
	c\geq a$ then
	$$
	b\sim c\iff b=c.
	$$
\end{lem}
\begin{proof}
	If $b=\Delta(b\join c, c)$ then we have 
	$a\le c$ and $a\le b= \Delta(b\join c, c)$ and so 
	$b\join c = a\join\Delta(b\join c, a)\le c\join c=c$. Likewise 
	$b\join c\le b$ and so $b=c$.
\end{proof}

A small variation of the proof shows that if $a\le b, c$ 
then $b\preceq c$ iff $b\le c$.

\begin{rem}\label{rem:one}
	Also from \cite{BO:eq} (lemma 2.7c for transitivity) we know that $\sim$ is an equivalence 
relation. In general it is not a congruence relation, but it does fit 
well with caret. 

It is clear that $\preceq$ induces a partial order on $\mathcal 
L/\sim$. Since $x\le y$ implies $x\preceq y$ we see that 
$x\mapsto\eqcl[x]$ is order-preserving. 

We will show that the structure $\mathcal L/\sim$ is an implication 
algebra -- with $\eqcl[x]\join\eqcl[y]=\eqcl[x\join\Delta(x\join y, y)]$ 
and $\eqcl[x]\meet\eqcl[y]=\eqcl[x\meet\Delta(x\join y, y)]$ whenever 
this exists -- and is an implication lattice iff $\mathcal L$ is an 
MR-algebra. 
\end{rem}

\begin{defn}\label{def:implColl}
	The poset $\mathcal L/\sim$ is the \emph{implication collapse} (or 
	just \emph{collapse}) of $\mathcal L$. 
	
	The mapping $\eta\colon\mathcal L\to\mathcal L/\sim$ given by 
	$$
	\eta(x)=\eqcl[x]
	$$
	is the \emph{collapsing} or the \emph{collapse mapping}. We 
	will often denote this mapping by $\mathcal L\mapsto\rsf 
	C(\mathcal L)$. 
\end{defn}

\subsection{Properties of the collapse}
The structure $\mathcal L/\sim $ is naturally an implication algebra. 
To show this we need to show that certain operations cohere with 
$\sim$. Before doing so we need to argue that most of our work can 
be done inside an interval algebra. The crucial tool is the following 
transfer theorem. 

\begin{thm}[Transfer]\label{thm:transfer}
	Let $\mathcal L$ be a cubic algebra and $a, b\in\mathcal L$. Then 
	$$
	a\sim b\text{ in }\mathcal L\iff a\sim b\text{ in }\env({\mathcal L}). 
	$$
	Furthermore, if $a\in\mathcal L$, $b\in\env(\mathcal L)$ and 
	$a\sim b$ then $b\in\mathcal L$. 
\end{thm}
\begin{proof}
	Since $\mathcal L$ is an upwards closed subalgebra of the MR-algebra
	$\env(\mathcal L)$. 
\end{proof}

The use of the transfer theorem is to allow us to prove facts about 
$\sim$ in a cubic algebra by proving them in an MR-algebra. But then 
we are actually working in a finitely generated sub-algebra of an 
MR-algebra which is isomorphic to an interval algebra. Thus we can 
always assume we are in an interval algebra. 

In some arbitrary cubic algebra $\mathcal 
L$ there are three operations to consider:
\begin{itemize}
	\item  $a\caret b$ -- will give rise to meets in $\mathcal L/\sim$; 

	\item 	$a*b=a\join\Delta(a\join b, b)$ -- 
		this operation will give rise to joins in $\mathcal L/\sim$;  

	\item 
	$a\Rightarrow b=\Delta(a\join b, a)\rightarrow b$ --
	this operation will give rise to implication in $\mathcal L/\sim$. 
\end{itemize}
We note that $a*b$ and $a\Rightarrow b$ are defined for any two elements in any cubic algebra.

Over any implication algebra the relation $\sim$ simplifies immensely. 
\begin{lem}\label{lem:simeq}
	Let $\brk<a, b>$ and $\brk<c, d>$ be in $\rsf I(\mathcal I)$. Then
	$$
	\brk<a, b>\sim \brk<c, d> \text{ iff }a\meet b=c\meet d. 
	$$
\end{lem}
\begin{proof}
	Suppose that $\brk<a, b>= \Delta(\brk<x, y>, \brk<c, d>) 
	=\brk<x\meet(y\to d), y\meet(x\to c)>$. Then 
	$x\meet(y\to d)\meet y\meet(x\to c)= 
	[x\meet(x\to c)]\meet[y\meet(y\to d)]= c\meet d$. 
	
	Conversely if $a\meet b=c\meet d$ we can do all computations in the 
	Boolean algebra $[c\meet d, 1]$ -- so that $\comp a\le b$ and $\comp 
	c\le d$ -- to get
	\begin{align*}
		\Delta(\brk<a, b>\join\brk<c, d>, \brk<c, d>) & =
		\Delta(\brk<a\join c, b\join d>)\\
		 & =\brk<(a\join c)\meet(\comp{b\join d}\join d ), 
		 (b\join d)\meet(\comp{a\join c}\join c>  \\
		 (a\join c)\meet(\comp{b\join d}\join d ) &=  
		 (a\join c)\meet(\comp b\join d)\\
		 & =(a\meet\comp b)\join(a\meet d)\join(c\meet\comp b)\join(c\meet d)  \\
		 & = \comp b\join (a\meet d)\join (c\meet\comp b)\join(c\meet d) \\
		 & = \comp b\join (a\meet d) \join(c\meet d)\\
		 & = \comp b\join (a\meet d) \join(a\meet b)  \\
		 & =\comp b\join(a\meet d)\join a\\
		 & = a. \\
		 (b\join d)\meet(\comp{a\join c}\join c ) &=  
		 (b\join d)\meet(\comp a\join c)\\
		 & =(b\meet\comp a)\join(b\meet c)\join(d\meet\comp a)\join(d\meet c)  \\
		 & = \comp a\join (b\meet c)\join (d\meet\comp a)\join(c\meet d) \\
		 & = \comp a\join (b\meet c) \join(c\meet d)\\
		 & = \comp a\join (b\meet c) \join(a\meet b)  \\
		 & =\comp a\join(b\meet c)\join b\\
		 & = b. \\
	\end{align*}
\end{proof}

We can restate the lemma by saying that 
$\iota\colon\brk<a, b>\mapsto a\meet b$ has the property 
that 
\begin{equation}
	\iota(\brk<a, b>)=\iota(\brk<c, d>)\text{ iff }\brk<a, b>\sim\brk<c, d>. 
	\label{eq:one}
\end{equation}
Thus for all $i\in\mathcal I$ we have 
$$
\iota(e_{\mathcal I}(i))=i
$$
so that $\iota$ is onto and $e_{\mathcal I}$ is a right inverse. 

Since we will often work in the intuitively clearer setting of 
Boolean algebras we will restate these results in that context. In 
this context the relation $\sim$ corresponds to a natural property 
of intervals -- the \emph{length}. 

\begin{defn}
	Let $x=[x_{0}, x_{1}]$ be any interval in a Boolean algebra $B$. 
	Then the \emph{length} of $x$ is $\comp{x_{0}}\meet x_{1}=\ell(x)$.
\end{defn}

\begin{cor}\label{cor:lenEq}
	Let $b, c$ be intervals in a Boolean algebra $B$. Then 
	$$
	b\sim c\iff\ell(b)=\ell(c).
	$$
\end{cor}
\begin{proof}
	We recall the isomorphism between the two definitions of $\rsf I(B)$ 
	given by 
	$$
	\brk<a, b>\mapsto[\comp a, b]. 
	$$
	Then we have 
	\begin{align*}
		\iota(\brk<a, b>) & =a\meet b  \\
		\ell([\comp a, b]) & =\comp{\comp a}\meet b  \\
		 & =a\meet b=\iota(\brk<a, b>). 
	\end{align*}
	The result is now immediate. 
\end{proof}

The remainder of the proof can be found in \cite{BO:UniMR} wherein 
we fully establish that $\mathcal L/\sim$ is an implication 
lattice with the following operations:
\begin{align*}
    \one &= [\one]\\
    \eqcl[a]\meet\eqcl[b] & =\eqcl[a\caret b]  \\
    \eqcl[a]\join\eqcl[b] & =\eqcl[a * b]  \\
    \eqcl[a]\rightarrow\eqcl[b] & =\eqcl[a\Rightarrow b];  
\end{align*}
and that 
this implication algebra is, locally, exactly the same as $\mathcal L$.

\begin{thm}
    On each interval $[a, \one]$ in $\mathcal L$ the mapping 
$x\mapsto\eqcl[x]$ is an implication embedding with upwards-closed 
range. 
\end{thm}

\section{Implication algebras to cubes}
In this section we develop a very general construction of cubic 
algebras. Although not every cubic algebra is isomorphic to one of 
this form (see \cite{BO:fil}) we will show in the next section that 
every cubic algebra is very close to to one of this form. We leave 
for later work a detailed analysis of exactly how close. 

Let $\mathcal I$ be an implication algebra. We define
$$
\rsf I(\mathcal I)=\Set{\brk<a, b> | a, b\in\mathcal I, a\join
b=\one\text{ and }a\meet b\text{ exists}}
$$
ordered by 
$$
\brk<a, b>\le\brk<c, d>\text{ iff }a\le c\text{ and }b\le d. 
$$
This is a partial order that is an upper semi-lattice with join 
defined by
$$
\brk<a, b>\join\brk<c, d>=\brk<a\join c, b\join d>
$$
and a maximum element $\one=\brk<1, 1>$. 

We can also define a $\Delta$ function by
$$
\text{if }\brk<c, d>\le\brk<a, b>\text{ then }
\Delta(\brk<a, b>, \brk<c, d>)=\brk<a\meet(b\to d), b\meet(a\to c)>. 
$$

We note the natural embedding of $\mathcal I$ into $\rsf I(\mathcal 
I)$ given by
$$
e_{\mathcal I}(a)=\brk<1, a>. 
$$

Note also that in an implication algebra $a\join b=\one$ iff $a\to 
b=b$ iff $b\to a=a$. 

Also $\Delta(\one, \bullet)$ is particularly simply defined as it is 
exactly $\brk<a, b>\mapsto\brk<b, a>$. 

We wish to show that the structure we have just described is a cubic algebra. 
We do this by showing that if $\mathcal I$ is a Boolean algebra 
then $\rsf I(I)$ is isomorphic to an interval algebra, and then use 
the fact that every interval in $I$ is a Boolean algebra and $\rsf 
I([a, \one])$ sits naturally inside $\rsf I(I)$. 

\begin{lem}\label{lem:Boolean}
	Let $B$ be a Boolean algebra. Then $\rsf I(B)$ is isomorphic to the 
	interval algebra of $B$. 
\end{lem}
\begin{proof}
	Let $\brk<a, b>\mapsto[\comp a, b]$. Since $a\meet b$ exists for all 
	$a, b\in B$ this imposes no hardship. The condition $a\to b=b$ is 
	equivalent to $\comp a\le b$. It is now clear that this mapping is a
	one-one, onto homomorphism. 
	
	We just check how the operations transfer:
	\begin{align*}
		\brk<a, b>\join\brk<c, d> & =\brk<a\join c, b\join d>  \\
		 & \mapsto[\comp a\meet\comp c, b\join d]  \\
		 & =[\comp a, b]\join[\comp c, d].   \\
		 \Delta(\brk<a, b>, \brk<c, d>) & = \brk<a\meet(b\to d), b\meet(a\to c)>   \\
		 & \mapsto[\comp a\join(b\meet\comp d), b\meet(\comp a\join c)]  \\
		 &  =\Delta([\comp a, b], [\comp c, d]). 
	\end{align*}
\end{proof}

Now to check that the axioms of a cubic algebra hold we just need to 
note that all of the axioms take place in some interval algebra -- 
since working above some $x=[u, v]\in\rsf I(\mathcal I)$ means that all the 
computations take place in the interval algebra $\rsf I([u\meet v, 
1])$ -- which we already know to be a cubic algebra. 

In fact we also have 
\begin{lem}\label{lem:intervals}
	$[\brk<a, b>, \brk<1, 1>]\sim[a\meet b, 1]$
\end{lem}
\begin{proof}
	Since $[\brk<a, b>, \brk<1, 1>]\sim [a, 1]\times[b, 1]\sim 
	[a\meet b, 1]$ by $\brk<c, d>\mapsto \brk<c, d>\mapsto c\meet d$. The 
	last is an isomorphism as it is an isomorphism of Boolean algebras 
	and in $[a\meet b, 1]$ the complement of $a$ is $b$. 
\end{proof}

\section{Some Category Theory}
The operation $\rsf I$ is a functor where we define $\rsf I(f)\colon\rsf 
I(\mathcal I_{1})\to\rsf I(\mathcal I_{2})$ by
$$
		\rsf I(f)(\brk<a, b>)=\brk<f(a), f(b)>	
$$
whenever $f\colon\mathcal I_{1}\to\mathcal I_{2}$ is an implication 
morphism. 

Since $f$ preserves all joins, implications and whatever meets exist 
we easily see that $\rsf I(f)$ is a cubic morphism. 

Clearly $\rsf I(fg)=\rsf I(f)\rsf I(g)$. 
The relation $\sim$ defined above
gives rise to a functor $\rsf C$ on cubic algebras. 
Before defining this we need a lemma. 
\begin{lem}\label{lem:simHom}
	Let $\phi\colon\mathcal L_{1}\to\mathcal L_{2}$ be a cubic 
	homomorphism. Let $a, b\in\mathcal L_{1}$. Then
	$$
		a\sim b\Rightarrow \phi(a)\sim\phi(b). 
	$$
\end{lem}
\begin{proof}
\begin{align*}
	a\sim b & \iff \Delta(a\join b, a)=b  \\
		 & \hphantom{\Leftarrow}\Rightarrow \phi(\Delta(a\join b, a))=\phi(b)  \\
		 & \iff \Delta(\phi(a)\join\phi(b), \phi(a))=\phi(b)  \\
		 & \iff \phi(a)\sim\phi(b). 
	\end{align*}
\end{proof}

Now $\rsf C$ is defined by
\begin{align*}
	\rsf C(\mathcal L) & =\mathcal L/\sim  \\
	\rsf C(\phi)([x]) & =[\phi(x)]. 
\end{align*}

It is easily seen that $\rsf C$ is a functor from the category of 
cubic algebras to the category of implication algebras. 

There are several natural transformations here. The basic ones are 
$e\colon\text{ID}\to\rsf I$ and $\eta\colon\text{ID}\to\rsf C$. 
These two are defined by 
\begin{align*}
	e_{\mathcal I}(x)&=\brk<\one, x>\\
	\eta_{\mathcal L}(x)&=[x]. 
\end{align*}
The commutativity of the diagram 
\begin{diagram}
	\mathcal I_{1} & \rTo^{\phi} & \mathcal I_{2}  \\
	\dTo^{e_{\mathcal I_{1}}} &  &  \dTo^{e_{\mathcal I_{2}}} \\
	\rsf I(\mathcal I_{1}) & \rTo_{\rsf I(\phi)} & \rsf I(\mathcal I_{2}) 
\end{diagram}
is from -- for $x\in\mathcal I_{1}$
\begin{align*}
	e_{\mathcal I_{2}}(\phi(x)) & =\brk<\one, \phi(x)>  \\
	 & =\brk<\phi(\one), \phi(x)>  \\
	 & =\rsf I(\phi)(\brk<\one, x>)\\
	 &=\rsf I(\phi)e_{\mathcal I_{1}}(x). 
\end{align*}

The commutativity of the diagram 
\begin{diagram}
	\mathcal L_{1} & \rTo^{\phi} & \mathcal L_{2}  \\
	\dTo^{\eta_{\mathcal L_{1}}} &  &  \dTo^{\eta_{\mathcal L_{2}}} \\
	\rsf C(\mathcal L_{1}) & \rTo_{\rsf C(\phi)} & \rsf C(\mathcal L_{2}) 
\end{diagram}
is from -- for $x\in\mathcal L_{1}$
\begin{align*}
	\eta_{\mathcal L_{2}}(\phi(x)) & =[\phi(x)]  \\
	 & =\rsf C(\phi)([x])\\
	 &=\rsf C(\phi)\eta_{\mathcal L_{1}}(x). 
\end{align*}

Then we get the composite transformation 
$\iota\colon\text{ID}\to\rsf C\rsf I$ defined by 
$$
	\iota_{\mathcal I}=\eta_{\rsf I(\mathcal I)}\circ e_{\mathcal I}. 
$$
By standard theory this is a natural transformation. It is easy to 
see that $e_{\mathcal I}$ is an embedding, and that $\eta_{\mathcal L}$ 
is onto. But there's more!
\begin{thm}\label{thm:isoIota}
	$\iota_{\mathcal I}$ is an isomorphism. 
\end{thm}
\begin{proof}
	Let $x, y\in\mathcal I$ and suppose that $\iota(x)=\iota(y)$. 
	Then 
	\begin{align*}
		\iota(x) & =\eta_{\rsf I(\mathcal I)}(e_{\mathcal I}(x))  \\
		 & =[\brk<\one, x>]  \\
		 & =[\brk<\one, y>]. 
	\end{align*}
	Thus $\brk<\one, x>\sim\brk<\one, y>$. Now 
	\begin{align*}
		\Delta(\brk<\one, x>\join\brk<\one, y>, \brk<\one, y>) & =
		\Delta(\brk<\one, x\join y>, \brk<\one, y>)\\
		 & =\brk<(x\join y)\to y, x\join y>. 
	\end{align*}
	This equals $\brk<\one, x>$ iff $x=x\join y$ (so that $y\le x$) and 
	$(x\join y)\to y=\one$ so that $y=x\join y$ and $x\le y$. Thus $x=y$. 
	Hence $\iota$ is one-one. 
	
	It is also onto, as if $z\in\rsf C\rsf I(\mathcal I)$ then we have 
	$z=[w]$ for some $w\in\rsf I(\mathcal I)$. But we know that 
	$w=\brk<x, y>\sim\brk<\one, x\meet y>$ -- since 
	$\Delta(\brk<\one, y>, \brk<\one, x\meet y>)=\brk<x, y>$ -- and so
	$z=[\brk<\one, x\meet y>]=\eta_{\rsf I(\mathcal I)}(e_{\mathcal 
	I}(x\meet y))$. 
\end{proof}

We note that there is also a natural transformation 
$\kappa\colon\text{ID}\to\rsf I\rsf C$ defined by 
$$
	\kappa_{\mathcal L}=e_{\rsf C(\mathcal L)}\circ \eta_{\mathcal L}. 
$$
In general this is not an isomorphism as there may be an MR-algebra $\mathcal 
M$ which is not a filter algebra, but $\rsf I(\rsf C(\mathcal M))$ is 
always a filter algebra. 

We also note that $\iota_{\rsf C(\mathcal L)}=\rsf C(\kappa_{\mathcal L})$ 
for all cubic algebras $\mathcal L$. 
The pair $\rsf I$ and $\rsf C$ do not form an adjoint pair.

\section{The range of $\rsf I$}

In this section
we wish to consider the relationship between $\mathcal L$ and $\rsf 
I(\mathcal L/\sim)$. In the case of $\mathcal L=\rsf I(\mathcal I)$,  
we saw in \thmref{thm:isoIota} that the two structures $I$ and $\rsf 
C(\mathcal L)$ are naturally 
isomorphic and that the set $e_{\mathcal I}[\mathcal I]\subseteq 
\rsf I(\mathcal I)$ has a very special place. This leads to the 
notion of \emph{g-cover}. 

%
%
%
%
%
%

\begin{defn}\label{def:gCover}
	Let $\mathcal L$ be a cubic algebra. Then $J\subseteq \mathcal L$ is 
	a \emph{g-cover} iff
	$J$ is an upwards-closed implication subalgebra and
	$$
	j\colon J\rInto\mathcal L\rTo^{\eta}\mathcal L/\sim
	$$
	is an isomorphism. 
	
	If $J$ is meet-closed we say that $J$ is a \emph{g-filter}. 
\end{defn}

We note that $\rsf I(\mathcal I)$ has a g-cover -- namely 
$e_{\mathcal I}[\mathcal I]$. We want to show that this is 
(essentially) the only way to get g-covers, and that having them 
simplifies the study of such second-order properties as congruences 
and homomorphisms.

If $J$ is a g-cover and $x\in\mathcal L$ then we have 
$x\sim j^{-1}(\eta(x))\in J$ and so $\leftGen J\rightGen =\mathcal L$. 
We need to be very precise about how $J$ generates $\mathcal L$ which 
leads to the next two lemmas. 

\begin{lem}\label{lem:simEq}
	Let $J$ be a g-cover for $\mathcal L$ and $x, y\in J$ with $x\sim y$. 
	Then $x=y$. 
\end{lem}
\begin{proof}
	If $x\sim y$ then $\eta(x)=\eta(y)$ and so $j(x)= j(y)$. 
	As $ j$ is one-one on $J$ this entails $x=y$. 
\end{proof}

\begin{lem}\label{lem:AlphaBeta}
	Let $J$ be a g-cover for $\mathcal L$ and $x\in \mathcal L$. There 
	exists unique pair $\alpha$, $\beta$ in $J$ with $\alpha\geq \beta$ 
	and $\Delta(\alpha, \beta)=x$. 
\end{lem}
\begin{proof}
	Let $x\in\mathcal L$. Then $\eta(x)\in\mathcal L/\sim= \rng(j)$. 
	Hence there is some $\beta\in J$ with $\eta(\beta)=\eta(x)$ and so 
	$\beta\sim x$. Let $\alpha=\beta\join x$. 
	
	If there is some other $\alpha'$ and $\beta'$ in $J$ with $\Delta(\alpha', 
	\beta')=x$ then $\beta'\sim x\sim \beta$ and so (by 
	\lemref{lem:simEq}) $\beta'=\beta$. Then $\alpha'=\beta'\join 
	x=\beta\join x= \alpha$. 
\end{proof}

\begin{thm}\label{thm:MRgCover}
	Suppose that $\mathcal M$ is an MR-algebra and $J$ is a g-cover. 
	Then $J$ is a filter -- in fact a g-filter by the above remarks. 
\end{thm}
\begin{proof}
	Let $x, y\in J$. Then we have $j(x\join y)= 
	j(x)\join j(y)= \eta(x)\join\eta(y)= \eta(x*y)$ so that
	$x\join y\sim x*y= x\join\Delta(x\join y, y)$. As 
	$(x\join y)\meet (x*y)$ exists this implies 
	$x\join y= x*y= x\join\Delta(x\join y, y)$ and so (by the MR-axiom)
	$x\meet y$ exists. Now let $w\in J$ be such that $w\sim (x\meet y)$. 
	Then there is some $x'\geq w$ with $x'\sim x$ and so 
	$x'=x$ as $x, x'\in J$. Likewise $w\le y$ and so $w\le x\meet y$ 
	i.e.  $w=x\meet y$ is in $J$. 
\end{proof}

\begin{rem}\label{rem:uSub}
	The above proof also shows us that if $J$ is a g-cover and $x, y\in 
	J$ are such that $x\meet y$ exists, then $x\meet y\in J$. 
	
	G-filters were considered in \cite{BO:fil} and used to get an 
	understanding of automorphism groups and the lattice of congruences. 
	G-covers generalize the notion of g-filters to a larger
	class of algebras, but we'll leave applications to second-order 
	properties to another paper. 
\end{rem}

Now suppose that $\mathcal L$ is any cubic algebra with a g-cover $J$. 
We want to show that $\rsf I(J)\sim \mathcal L$. 
For each $x\in \mathcal L$ there is a unique pair
$\alpha(x), \beta(x)$ in $J$ such that $\beta(x)\le\alpha(x)$ and
$x= \Delta(\alpha(x), \beta(x))$. Define 
$$
\phi\colon\mathcal L\to\rsf I(J)
$$
by
$$
\phi(x)=\brk<\alpha(x), \alpha(x)\to\beta(x)>. 
$$
We need to show that this is one-one, onto and order-preserving. 

We 
first note that $\alpha(x)\to\beta(x)=\Delta(\one, x)\join\beta(x)$. 
Since $x\sim\beta(x)$ we have $\beta(x)=(\Delta(\one, 
x)\join\beta(x))\meet(x\join\beta(x))$ and trivially 
$\one=(\Delta(\one, 
x)\join\beta(x))\join(x\join\beta(x))$.  Hence the complement of 
$\alpha(x)=x\join\beta(x)$ over $\beta(x)$ must be
$\alpha(x)\to\beta(x)=\Delta(\one, x)\join\beta(x)$. 

\begin{description}
	\item[One-one] Suppose that $\phi(x)=\phi(y)$. Then we have 
	\begin{align*}
		\alpha(x)\to\beta(x) & =\alpha(y)\to\beta(y)  \\
		\alpha(x) &  =\alpha(y) \\
		\intertext{Therefore}		
		\beta(x)=(\alpha(x)\to\beta(x))\meet\alpha(x)&=(\alpha(y)\to\beta(y))\meet\alpha(y)=\beta(y)  \\
		\intertext{and so we have }
		x=\Delta(\alpha(x), \beta(x)) & =\Delta(\alpha(y), \beta(y)) =y. 
	\end{align*}

	\item[Onto] Let $\brk<a, b>\in\rsf I(J)$. Let $z=\Delta(a, a\meet b)$. 
	Then we have -- by uniqueness -- that $\alpha(z)=a$ and 
	$\beta(z)=a\meet b$ and so
	$a\to (a\meet b)= a\to b= b$ -- by definition of $\rsf I(J)$. 

	\item[Order-preserving] Suppose that $x\le y$. Then we have 
	$x\sim \beta(x)$ and so there is some $b\geq \beta(x)$ with 
	$b\sim y$. As $b\in J$ we get $\beta(y)=b$. Hence
	$\alpha(x)=x\join\beta(x)\le y\join\beta(y)=\alpha(y)$. 
	Also (as $x\le y$) $\Delta(\one, x)\le\Delta(\one, y)$ and so
	$\Delta(\one, x)\join\beta(x)\le \Delta(\one, y)\join\beta(y)$. 
\end{description}

Thus we have 
\begin{thm}\label{thm:existsGCovers}
	A cubic algebra $\mathcal L$ has a g-cover iff $\mathcal L$ is 
	isomorphic to $\rsf I(\mathcal I)$ for some implication algebra 
	$\mathcal I$. 
\end{thm}

It follows from the above theorems that not every cubic algebra has a 
g-cover -- as we know that MR-algebras not isomorphic to 
filter algebras may exist (under certain set-theoretic assumptions) 
-- see \cite{BO:fil} section 6.

\section{Env and g-covers}
In this section we consider the relationship between g-covers in a 
cubic algebra and in its envelope. We discover that g-covers go 
downwards and upwards -- ie one has a g-cover iff the other has one. 
\begin{thm}\label{thm:envGCoversIII}
	Let $\mathcal L$ be a cubic algebra and suppose that $\rsf I(\rsf F)$ is a filter algebra
	and $\env(\mathcal 
	L)\rTo^{\sim}_{\phi}\rsf I(\rsf F)$ is a cubic homomorphism with 
	upwards-closed range. Then 
	the homomorphism restricts to $\mathcal L$ as -- 
	\begin{diagram}
		\mathcal L & \rDotsto^{\phi\restrict\mathcal L} & \rsf I(\rsf 
		F\Cap\mathcal L)  \\
		\dInto &  & \dInto_{\rsf I(\text{incl})}  \\
		\env(\mathcal L) & \rTo_{\phi} & \rsf I(\rsf F)
	\end{diagram}
	where $\rsf F\Cap\mathcal L=\Set{\phi(l) | l\in\mathcal L\text{ and 
	}\phi(l)\in\rsf F}$. 
\end{thm}
\begin{proof}
	We first note that $\rsf F\Cap\mathcal L$ is an implication algebra 
	as $\phi[\mathcal L]$ and $\rsf F$ are implication subalgebras of 
	$\rsf I(\rsf F)$. 
	\begin{enumerate}[{Claim }1:]
		\item If $l\in\mathcal L$ then $\phi(l)\in\rsf I(\rsf F\Cap\mathcal L)$.
		
		$\mathcal L$ is upwards closed in $\env(\mathcal L)$ and so 
		$\phi[\mathcal L]$ is upwards closed in $\rsf I(\rsf F)$. Thus 
		$\rsf F\Cap\mathcal L$ is an upper segment of $\rsf F$. 
		
		Let $l\in\mathcal L$. Then $\phi(l)\in\rsf I(\rsf F)$ and so there 
		is some $l'\in\rsf F\Cap\mathcal L$ so that $l'\sim\phi(l)$. Then 
		$\phi(l)\join l'\in\rsf F\Cap\mathcal L$ and so
		$\phi(l)= \Delta(\phi(l)\join l', l')\in\rsf I(\rsf F\Cap\mathcal L)$. 
	
		\item $\phi\restrict L$ is onto $\rsf I(\rsf F\Cap\mathcal L)$. 
		
		If $x\in\rsf I(\rsf F\Cap\mathcal L)$ then we can find some 
		$x'\in\rsf F\Cap\mathcal L$ so that $x\sim x'$. By definition 
		$x'=\phi(l)$ for some $l\in\mathcal L$ and as $x\join\phi(l)\in\rsf 
		F\Cap\mathcal L$ there is also some $m\in\mathcal L$ with $\phi(m)= 
		x\join\phi(l)$. Now we have 
		$x= \Delta(x\join\phi(l), \phi(l))= \Delta(\phi(m), \phi(l))= 
		\phi(\Delta(m, l))$ is in the range of $\phi\restrict\mathcal L$. 
	\end{enumerate}
\end{proof}

\begin{cor}\label{cor:filter}
	If $\env(\mathcal L)$ is isomorphic to a filter algebra then $\mathcal 
	L$ has a g-cover. 
\end{cor}
\begin{proof}
	Let $\phi\colon\env(\mathcal L)\to\rsf I(\rsf F)$ be the isomorphism. 
	Then $\phi\restrict\mathcal L$ is also an isomorphism -- it is 
	one-one as it is the restriction of a one-one function, and onto by 
	the theorem. Since $\rsf I(\rsf F\Cap\mathcal L)$ has a g-cover, so 
	does $\mathcal L$. 
\end{proof}

The above results show that g-covers go down to certain subalgebras. 
Now we look at making them go up. 

\begin{thm}\label{thm:envGCoverII}
	Let $\mathcal L$ be a cubic algebra and suppose that $J$ is a 
	g-cover for $\mathcal L$. Then $J$ has fip in $\env(\mathcal L)$ and 
	the filter it generates is a  g-filter. 
\end{thm}
\begin{proof}
	This is very like the proof to \thmref{thm:MRgCover}. 
	Let $x, y\in J$. Then we have $j(x\join y)= 
	j(x)\join j(y)= \eta(x)\join\eta(y)= \eta(x*y)$ so that
	$x\join y\sim x*y= x\join\Delta(x\join y, y)$. As 
	$(x\join y)\meet (x*y)$ exists this implies 
	$x\join y= x*y= x\join\Delta(x\join y, y)$. Thus in $\env(\mathcal 
	L)$ the meet $x\meet y$ exists. By earlier work 
	(\cite{BO:fil}, 
	Lemma 19) this implies $J$ has 
	fip in $\env(\mathcal L)$. 
	
	Let $\rsf F$ be the filter generated by $J$. 
	
	Now if $z\in\env(\mathcal L)$ we have $x_{1}, \dots, 
	x_{k}\in\mathcal L$ such that 
	$x_{1}\caret(x_{2}\caret(\dots \caret x_{k}))=z$. Let $y_{i}\in J$ be 
	such that $x_{i}\sim y_{i}$. Then 
	$y_{1}\meet\dots\meet y_{k}\preccurlyeq z$ and so $z\in\leftGen\rsf 
	F\rightGen$. 
\end{proof}

\begin{cor}\label{cor:downGC}
	Let $\mathcal L$ be an upwards-closed cubic subalgebra of a cubic algebra $\mathcal M$ 
	with g-cover $J$. Then $\mathcal L$ has a g-cover. 
\end{cor}
\begin{proof}
	Let $J$ be as given and let $\hat J$ be the extension to a g-filter 
	for $\env(\mathcal M)$. Then we have 
	\begin{diagram}
		\mathcal L &  &  & \rDotsto^{\phi\restrict\mathcal L} &  &  & \rsf 
		I(\hat J\Cap\mathcal L)  \\
		 & \rdInto &  &  &  &  &   \\
		\dInto &  & \mathcal M & \rTo^{\sim} & \rsf I(J) &  & \dTo  \\
		 &  & \dInto &  &  & \rdTo &   \\
		\env(\mathcal L) & \rInto & \env(\mathcal M) &  & \rTo_{\sim} &  & \rsf 
		I(\hat J)
	\end{diagram}
	Since $\phi\restrict\mathcal L$ is one-one and onto we have the result. 
%
\end{proof}

\begin{rem}\label{rem:gCoverDown}
	By a slightly different argument we can show that if
	$\mathcal L$ is an upwards-closed cubic subalgebra of a cubic algebra $\mathcal M$ 
	with g-cover $J$, then $\mathcal L\cap J$ is a g-cover for $\mathcal L$.
\end{rem}

\begin{defn}\label{def:ccPres}
	A cubic algebra $\mathcal L$ is \emph{countable presented} iff there 
	is a countable set $A\subseteq\mathcal L$ such that 
	$\mathcal L=\bigcup_{a\in A}\mathcal L_{a}$. 
\end{defn}

It is easy to show that if $\mathcal L$ is countably presented, then 
so is $\env(\mathcal L)$. 
It then follows from the fact that every countably presented MR-algebra is 
a filter algebra that every countably presented cubic algebra has  
a g-cover. 

Another interesting consequence for implication algebras is 
\begin{thm}\label{thm:inplEmbed}
	Let $\mathcal I$ be an implication algebra. Then $\mathcal I$ is 
	isomorphic to an upper segment of a filter. 
\end{thm}
\begin{proof}
	Consider 
	$$
	p\colon\mathcal I\rTo^{e_{\mathcal I}}\rsf I(\mathcal 
	I)\rInto\env(\rsf I(\mathcal I))\rTo^{\eta}\env(\rsf I(\mathcal 
	I))/\sim. 
	$$
	Then $p$ is an implication morphism as each component is one, and it 
	is easy to see that the range of $p$ is upwards closed. 
	We want to see that $p$ is one-one:
	\begin{align*}
		p(x)=p(y) & \rightarrow \eta(\text{incl}(e_{\mathcal I}(x)))= 
		\eta(\text{incl}(e_{\mathcal I}(y)))\\
		 &\rightarrow \text{incl}(e_{\mathcal I}(x))\sim \text{incl}(e_{\mathcal 
		 I}(y))   \\
		 & \rightarrow e_{\mathcal I}(x)\sim e_{\mathcal I}(y). 
	\end{align*}
	This implies $x=y$ since 
	if $e_{\mathcal I}(x)=\brk<1, x>\sim e_{\mathcal I}(y)=\brk<1, y>$ 
	then 
	$\brk<1, x>= \Delta(\brk<1, x\join y>, \brk<1, y>)= \brk<(x\join 
	y)\to y, x\join y>$
	and so $x\join y=x$ (and therefore $y\le x$) and $1=(x\join 
	y)\to y$ (and therefore $x\join y\le y$ i.e. $x\le y$). Thus $x=y$. 
\end{proof}

The filter obtained by this theorem sits over $\mathcal I$ in a way 
similar to the way $\env(\mathcal L)$ sits over $\mathcal L$. For that 
reason we will also call this an \emph{enveloping lattice} for an 
implication algebra and denote it by $\env(\mathcal I)$. The next 
theorem is clear. 

\begin{thm}\label{thm:envImpl}
	Let $\mathcal L$ be a cubic algebra with g-cover $J$. Then 
	$\env(J)$ is isomorphic to $\env(\mathcal L)/\sim$ and the 
	following diagram commutes:
	$$
	\begin{diagram}
		\mathcal L & & \rTo^{\sim} &  &  \rsf I(J)  \\
		\dInto &  & & &\dTo    \\
		\env(\mathcal L) & \rTo_{\sim} & \rsf I(\env(\mathcal L)/\sim)& \rTo_{\rsf I(\sim)} & \rsf I(\env(J)) 
	\end{diagram}
	$$
\end{thm}

Now we consider the last step in the puzzle -- the relationship 
between $\mathcal L$ and $\rsf I(\mathcal L/\sim)$. Clearly they 
collapse to the same implication algebra. From \corref{cor:downGC} we 
know that if $\mathcal L$ has no g-cover then we cannot embed $\mathcal 
L$ as an upwards-closed subalgebra of $\rsf I(\mathcal L/\sim)$.

Embedding it as a subalgebra seems possible but we have no idea how 
to do it.

%
%

\section{An algebra of covers}
In this section we consider the family of all g-covers of a cubic algebra and deduce an interesting 
MR-algebra. This section is very like similar material on filters --
see \cites{BO:fil, BO:filAlg}. Therein we showed the following results on
finite intersection property.

\begin{lem}\label{lem:fipPreFilter}
	Let $\rsf I(B)$ be an interval algebra and $A\subseteq\rsf I(B)$. 
	Then $A$ has fip iff for all $x, y\in A$ $x\meet y$ exists. 
\end{lem}

\begin{defn}\label{def:compatible}
	Let $\mathcal L$ be a cubic algebra and $A\subseteq\mathcal L$. 
	$A$ is \emph{compatible} iff for all embeddings $e\colon{\mathcal 
	L}\to{\rsf I(B)}$, the set $e[A]$
	has fip. 
\end{defn}

\begin{cor}\label{cor:preFilterFIP}
	Let $\mathcal L$ be a cubic algebra and $A\subseteq\mathcal L$. 
	Then $A$ is compatible iff for all $x, y\in A$\ $x\join\Delta(\one, y)=\one$. 
\end{cor}

For later we have the following useful lemma relating compatibility and the $\preceq$ relation.
\begin{lem}\label{lem:compEQ}
    If $x\preceq y$ and $x\join\Delta(\one,  y)=\one$ then $x\le y$.
\end{lem}
\begin{proof}
    $x\preceq y$ implies $y=(y\join x)\meet(y\join\Delta(1, x))= y\join x$ as the latter term is $\one$.
    Thus $x\le y$.
\end{proof}

Our interest is in a special class of upwards-closed implication
subalgebras. 

\begin{defn}\label{def:sia}
    A \emph{special subalgebra} of a cubic algebra
    $\mathcal L$ is an upwards-closed implication
    subalgebra $I$ that is compatible and for all $x, y\in I$, if
    $x\meet y$ exists in $\mathcal L$ then $x\meet y\in I$.
\end{defn}

\begin{lem}\label{lem:gcEQcomp}
    Every g-cover is special.
\end{lem}
\begin{proof}
    Let $J$ be a g-cover. As noted in \ref{rem:uSub} the second condition holds. 
    
    Compatibility follows from \thmref{thm:envGCoverII}.
\end{proof}

\begin{lem}\label{lem:intersectSpec}
    Let $\bbm I$ be a family of special subalgebras. Then
    $\bigcap\bbm I$ is also special.
\end{lem}
\begin{proof}
    Immediate.
\end{proof}

This lemma implies that any compatible set is contained in some
smallest special subalgebra.

Now we need to define some operations on special subalgebras.

\begin{lem}\label{lem:interTwoFil}
	Let $\mathcal L$ be a cubic algebra and $\rsf I$ and $\rsf J$ be two 
	special subalgebras. Then
	$$
	\rsf I\cap\rsf J=\Set{f\join g | f\in\rsf I\text{ and }g\in\rsf G}. 
	$$
\end{lem}
\begin{proof}
	The RHS set is clearly a subset of both $\rsf I$ and $\rsf J$. 
	
	And if $z\in\rsf I\cap\rsf J$ then $z=z\join z$ is in the RHS set. 
\end{proof}

\begin{defn}\label{def:cup}
	Let $\rsf I, \rsf J$ be two special subalgebras of $\mathcal L$. Then
	$\rsf I\vee\rsf J$ is defined iff $\rsf I\cup\rsf J$ is
	compatible, in 
	which case it is the special subalgebra generated by $\rsf I\cup\rsf J$. 	
\end{defn}

\begin{lem}\label{lem:cup}
	If $\rsf I\vee\rsf J$ exists then it is equal to
	$\Set{f\meet g | f\in\rsf I\text{ and }g\in\rsf J\text{ and
	}f\meet g\text{ exists}}$. 
\end{lem}
\begin{proof}
	Let $S$ be this set. It is clearly contained in $\rsf I\join\rsf J$. 
	
	To show the converse we need to show that $S$ is a special
	subalgebra. Recall that $\rsf I\cup\rsf J$ is assumed to be
	compatible. 
	\begin{description}
	    \item[Upwards-closure] if $h\geq f\meet g$ for $f\in\rsf I$ 
	    and $g\in\rsf J$ then $h=(h\join f)\meet(f\join g)$ is
	    also in $S$.
	    
	    \item[$\to$-closure] follows from upwards-closure.
	    
	    \item[Compatible] if $a\meet b\in S$ and $f\meet g\in S$
	    with $a, f\in\rsf I$ and $b, g\in\rsf J$ then 
	    $a$ is compatible with both $f$ and $g$ so that 
	    $1=a\join\Delta(\one, f)= a\join \Delta(\one, g)$, whence 
	    $\one= a\join\Delta(\one,  f\meet g)$. Likewise
	    $\one= b\join\Delta(\one,  f\meet g)$ so that 
	    $\one= (a\meet b)\join\Delta(\one,  f\meet g)$.
	    
	    \item[All available intersections] if $a\meet b\in S$ and $f\meet g\in S$
	    with $a, f\in\rsf I$ and $b, g\in\rsf J$ and 
	    $(a\meet b)\meet(f\meet g)$ exists in $\mathcal L$, then 
	    $s=a\meet f\in\rsf I$ and $t=b\meet g\in\rsf J$ and
	    $s\meet t$ exists,  so that $s\meet t\in S$.
	 \end{description}
\end{proof}

It is easy to show that these operations are commutative, associative,  
idempotent
and satisfy absorption. 
Distributivity also holds in a weak way. 
\begin{lem}\label{lem:distrib}
	Let $\rsf I, \rsf J, \rsf K$ be special subalgebras of a
	special subalgebra $\rsf S$. Then
	$$
	\rsf I\cap(\rsf J\join\rsf K)=(\rsf I\join\rsf J)\cap(\rsf I\join\rsf 
	K). 
	$$
\end{lem}
\begin{proof}
    	As everything sits inside the compatible set $\rsf S$ there
	are no issues of incompatibility. 
    
	Let $x=g\join(h\meet k)\in\rsf I\cap(\rsf J\join\rsf K) $. 
	Then $x=(g\join h)\meet(g\join k)$ is in 
	$(\rsf I\join\rsf J)\cap(\rsf I\join\rsf J)$. 
	
	Conversely if $x=(g_{1}\join h)\meet(g_{2}\join k)$ is in 
	$(\rsf I\join\rsf J)\cap(\rsf I\join\rsf K)$ then 
	$g_{1}\join k\geq g_{1}\in\rsf I$ and $g_{2}\join h\geq
	g_{2}\in\rsf I$ and the meet exists, so $x\in\rsf I$.
	Also $g_{1}\join k\geq k\in\rsf J$ and $g_{2}\join h\geq
	h\in\rsf K$ so that $x\in\rsf J\join \rsf K$.
\end{proof}

\subsection{Near-principal}
There is a very special case of special subalgebra that merits attention,  as it 
leads into the general theory so well,  \emph{principal subalgebras}. These are of the form
$[g, \one]$ for some $g\in\mathcal L$. It is easy to verify that these are special.

Also associated with elements of $\mathcal L$ is an operation on special subalgebras.
Suppose that $\rsf I$ is a special subalgebra.
\begin{lem}\label{lem:gFF}
	The set
	$$
	\rsf I_{g}=\Set{\Delta(g\join f, f) | f\in\rsf I}
	$$
	is compatible and upwards closed. 
\end{lem}
\begin{proof}
	We just need to check this for intervals. Suppose that
	$g=[g_{0}, g_{1}]$, $f_{0}=[x, y]\in\rsf I$ and $f_{1}=[s, t]\in\rsf I$. Then
	\begin{align*}
	\Delta(g\join f_{0}, f_{0})&=[(g_{0}\meet x)\join(g_{1}\meet\comp y), 
	(g_{0}\join\comp x)\meet(g_{1}\join y)]\\
	\Delta(\one, \Delta(g\join f_{1}, f_{1})&=[(\comp g_{0}\meet s)\join(\comp g_{1}\meet\comp t), 
		(\comp g_{0}\join\comp s)\meet(\comp g_{1}\join t)]\\
	\intertext{Thus }
	\Delta(g\join f_{0}, f_{0})\join&\Delta(\one,  \Delta(g\join f_{1}, f_{1})\\
	[\bigl((g_{0}\meet x)\join(g_{1}\meet\comp y)\bigr)&\meet\bigl((\comp g_{0}\meet s)\join(\comp g_{1}\meet\comp t)\bigr), \\
	&\qquad \bigl((g_{0}\join\comp x)\meet(g_{1}\join y)\bigr)\join\bigl((\comp g_{0}\join\comp s)\meet(\comp g_{1}\join t)\bigr)]\\
	&= [(g_{0}\meet x\meet\comp g_{0}\meet s)\join(g_{0}\meet x\meet\comp g_{1}\meet\comp t)\join(g_{1}\meet\comp y\meet\comp g_{0}\meet s)\join(g_{1}\meet\comp y\meet \comp g_{1}\meet\comp t), \\
	&\qquad (g_{0}\join\comp x\join\comp g_{0}\join\comp s)\meet(g_{0}\join\comp x\join\comp g_{1}\join t)\meet(g_{1}\join y\join\comp g_{0}\join\comp s)\meet(g_{1}\join y\join\comp g_{1}\join t)]\\
	&=[0, 1]\\
	\intertext{since $f_{0}$ and $f_{1}$ are compatible and so}
	[0, 1]=f_{0}\join\Delta(\one) &= [x\meet\comp t,  y\join\comp s].
	\end{align*}
	
	To show upwards closure we note that if $k\geq\Delta(g\join f, f)$ 
	for some $f\in\rsf I$ then we have $k\in\leftGen\rsf I\rightGen$ and 
	so there is some $k'\in\rsf I$ with $k\sim k'$. Then we have 
	$\Delta(g\join k', k')\sim k\geq \Delta(g\join f, f)$. This implies
	$k$ and $\Delta(g\join k', k')$ are compatible,  and therefore equal.
\end{proof}

\begin{lem}\label{lem:interSect}
	$$\rsf I\cap\rsf I_{g}=[g, \one]\cap\rsf I. $$
\end{lem}
\begin{proof}
	If $f\in\rsf I\cap[g, \one]$ then $g\join f=f$ and so
	$\Delta(g\join f, f)=\Delta(f, f)= f\in\rsf I_{g}$. 
	
	Conversely, if $h\in\rsf I\cap\rsf I_{g}$ then we have 
	$h$ and $\Delta(g\join h, h)$ are compatible and so $h= \Delta(g\join h, h)$. 
	Therefore $g\join h=h$ and $g\le h$. 
\end{proof}

\begin{thm}\label{thm:gFF}
	The set
	$$
	\rsf I_{g}=\Set{\Delta(g\join f, f) | f\in\rsf I}
	$$
	is a special subalgebra and $\leftGen\rsf I_{g}\rightGen=\leftGen\rsf I\rightGen$. 
\end{thm}
\begin{proof}
	That $\rsf I_{g}$ is compatible and upwards-closed follows from the lemma.
	If 
	$f_{1}, f_{2}\in\rsf I$ and the meet $\Delta(g\join f_{1}, f_{1})\meet \Delta(g\join f_{2}, f_{2})$ exists.
	Let $h_{i}=\Delta(g\join f_{i}, f_{i})$.
	
	In any interval algebra, if $f_{1}\meet f_{2}$ exists, then 
	$\Delta((g\join f_{1})\meet(g\join f_{2}),  f_{1}\meet f_{2})= h_{1}\meet h_{2}$.
	
	In this case, we know that $h_{1}\meet h_{2}$ exists,  and so $(g\join f_{1})\meet (g\join f_{2})$ exists.
	This is therefore in $\rsf I$ as both factors are. As it is also in $[g, \one]$ it
	is in $\rsf I_{g}$. 
	From our remark concerning interval algebras we see that 
	$\Delta((g\join f_{1})\meet(g\join f_{2}),  h_{1}\meet h_{2})$ is below both $f_{1}$ and
	$f_{2}$ so that it must equal it in $\rsf I$. The same formula shows that $h_{1}\meet h_{2}$ is 
	in $\rsf I_{g}$.
	
	By definition, 
	for each $f\in\rsf I$ there is a $f'\in\rsf I_{g}$ such that
	$f\sim f'$, and conversely. Thus 
	$\leftGen\rsf I_{g}\rightGen=\leftGen\rsf I\rightGen$.
\end{proof}

Note that a special case of this is when $g=\one$ and we have
$\rsf I_{\one}=\Delta(\one, \rsf I)$ and that for a principal filter
$[h, \one]$ we have $[{h, \one}]_{g}=[\Delta(g\join h, h), \one]$. 

\begin{cor}\label{cor:implSS}
    The set
    $$
    g\to \rsf I=\Set{g\to f | f\in\rsf I}
    $$
    is a special subalgebra.
\end{cor}
\begin{proof}
	Recall that $g\to f=\Delta(\one, \Delta(g\join f, f))\join f$. 
	Hence if
	$\rsf J=\Delta(\one, \rsf I_{g})$ then
	\begin{align*}
		\rsf J\cap\rsf I & =\Set{f\join\beta_{\rsf J}(f) | f\in\rsf I}  \\
		 & =\Set{\Delta(\one, \Delta(g\join f, f))\join f | f\in\rsf I}  \\
		 & =\Set{g\to f | f\in\rsf I}. 
	\end{align*}
\end{proof}

\begin{cor}\label{cor:gFilter}
	If $g\in\rsf F$ then 
	$$\rsf F\cap\rsf F_{g}=[g, \one]. $$
\end{cor}
\begin{proof}
	Obvious
\end{proof}

Interestingly enough the converse of \lemref{lem:interSect} is also true. 
\begin{lem}\label{lem:gFilter}
	Suppose that $\leftGen\rsf J\rightGen=\leftGen\rsf I\rightGen$
	and $\rsf I\cap\rsf J=[g, \one]$. Then
	$$\rsf J=\rsf I_{g}.$$ 
\end{lem}
\begin{proof}
	Clearly $[g, \one]\subseteq\rsf J$. 
	
	For arbitrary $h\in\rsf J$ we can find $f\in\rsf I$ and $h'\in\rsf I_{g}$ with
	$\Delta(g\join f, f)=h'\sim h$. Then $h'\join f= g\join f$.
	
	Also $h\join f\in\rsf I\cap\rsf J$ and so $g\le h\join f$. Now $h\sim h'\le g\join f\in\rsf J$
	implies $h\le g\join f$ also. Thus $g\join f= h\join f= h'\join f$.
	
	As $f\sim h\sim h'$ we have $h'=\Delta(h'\join f, f)= \Delta(h\join f, f)= h$.
	
	Thus $\rsf J\subseteq\rsf I_{g}$.
	
	The reverse implication follows as 
	$\leftGen\rsf J\rightGen= \leftGen\rsf I\rightGen= \leftGen\rsf I_{g}\rightGen$
	and so if $h\in\rsf I_{g}$ there is some $h'\in\rsf J$ with $h\sim 
	h'$. As $h$ and $h'$ are compatible (as $\rsf J\subseteq\rsf I_{g}$) 
	we have $h=h'\in\rsf J$. 
\end{proof}

\begin{cor}\label{cor:Idempotence}
	Let $g, h\in\rsf I$. Then 
	\begin{enumerate}[(a)]
		\item $\rsf I=(\rsf I_{g})_{g}$; 
	
		\item $(\rsf I_{g})_{h}=(\rsf I_{g})_{g\join h}$. 
	\end{enumerate}
\end{cor}
\begin{proof}
	\begin{enumerate}[(a)]
		\item Since $\rsf I\cap\rsf I_{g}=[g, \one]$and 
		$\leftGen\rsf I\rightGen= \leftGen\rsf I_{g}\rightGen$ the lemma 
		implies $\rsf I=(\rsf I_{g})_{g}$. 
	
		\item 
		\begin{align*}
			\rsf I_{g}\cap(\rsf I_{g})_{h} & = [h, \one]\cap\rsf I_{g}  \\
			 &  =[h, \one]\cap\rsf I\cap\rsf I_{g} \\
			 &  =[h, \one]\cap[g, \one] \\
			 &  =[h\join g, \one]. 
		\end{align*}
		The lemma now implies $(\rsf I_{g})_{h}=(\rsf I_{g})_{g\join h}$. 
	\end{enumerate}
\end{proof}

\subsection{Relative Complements}
Let
$\rsf J\subseteq\rsf I$ be two special subalgebras. There are 
several ways to define the relative complement of $\rsf J$ in $\rsf I$. 

\begin{defn}\label{def:impl}
	Let $\rsf J\subseteq\rsf I$ be two special subalgebras. Then
	\begin{enumerate}[(a)]
		\item $\rsf J\supset\rsf I=\bigcap\Set{\rsf H | \rsf H\join\rsf J=\rsf I}$; 
	
		\item $\rsf J\Rightarrow\rsf I=\bigvee\Set{\rsf H |\rsf H\subseteq\rsf 
		F\text{ and }\rsf H\cap\rsf J={\Set{\one}}}$; 
	
		\item $\rsf J\to\rsf I=\Set{h\in\rsf I |\forall g\in\rsf J\ h\join 
		g=\one}$. 
	\end{enumerate}
\end{defn}

We will now show that these all define the same set. 

\begin{lem}\label{lem:twoThreeSame}
	$\rsf J\to\rsf I=\rsf J\Rightarrow\rsf I$. 
\end{lem}
\begin{proof}
	Let $h\in(\rsf J\to\rsf I)\cap\rsf J$. Then $\one=h\join h=h$. Thus
	$\rsf J\to\rsf I\subseteq\rsf J\Rightarrow\rsf I$. 
	
	Suppose that $\rsf H\subseteq\rsf I$ and $\rsf H\cap\rsf J=\Set{\one}$. 
	Let $h\in\rsf H$ and $g\in\rsf J$. Then
	$h\join g\in\rsf H\cap\rsf J=\Set{\one}$ so that $h\join g=\one$. 
	Hence $\rsf H\subseteq(\rsf J\to\rsf I)$ and so
	$\rsf J\Rightarrow\rsf I\subseteq\rsf J\to\rsf I$. 
\end{proof}

\begin{lem}\label{lem:smallH}
	Let $h\in\rsf I$ and $g\in\rsf J$ be such that $g\join h<\one$. Then
	$h\notin g\to\rsf I$. 
\end{lem}
\begin{proof}
	This is clear as $h=g\to f$ implies $h\join g=\one$. 
\end{proof}

\begin{thm}\label{thm:oneThreeEqual}
	$\rsf J\supset\rsf I=\rsf J\to\rsf I$. 
\end{thm}
\begin{proof}
	Suppose that $h\notin \rsf J\to\rsf I$ so that there is some 
	$g\in\rsf J$ with $h\join g<\one$. Then $h\notin g\to\rsf I$ and 
	clearly $\rsf I=[g, \one]\join(g\to\rsf I)$ so that $\rsf 
	G\supset\rsf I\subseteq g\to\rsf I$ does not contain $h$. Thus
	$\rsf J\supset\rsf I\subseteq\rsf J\to\rsf I$. 
	
	Conversely if $\rsf H\join\rsf J=\rsf I$ and $k\in\rsf J\to\rsf I$ 
	then there is some $h\in\rsf H$ and $g\in\rsf J$ with $k=h\meet g$. 
	But then 
	\begin{align*}
		k & =k\join(h\meet g)  \\
		 & =(k\join h)\meet(k\join g)  \\
		 & =k\join h && \text{ as }k\join g=\one
	\end{align*}
	and so $k\geq h$ must be in $\rsf H$. 
	Thus $\rsf J\to\rsf I \subseteq\rsf J\supset\rsf I$. 
\end{proof}

We earlier defined a filter $g\to\rsf I$. We now show that this new 
definition of $\to$ extends this earlier definition. 

\begin{lem}\label{lem:implGG}
	Let $g\in\rsf I$. Then 
	$$
	g\to\rsf I=[g, \one]\to\rsf I. 
	$$
\end{lem}
\begin{proof}
	Let $g\to f\in g\to\rsf I$ and $k\in[g, \one]$. Then
	$k\join(g\to f)\geq g\join (g\to f)=\one$. Thus $g\to f\in[g, 
	\one]\to\rsf I$ and so $g\to \rsf I\subseteq[g, 
	\one]\to\rsf I$. 
	
	Conversely, if $h\in[g, \one]\to\rsf I$ then $h\join g=\one$ and so
	$h$ is the complement of $g$ in $[h, \one]$. Thus
	$h=g\to h\in g\to\rsf I$ and so $[g,\one]\to\rsf I\subseteq
	g\to \rsf I$. 
\end{proof}

\begin{lem}\label{lem:inclImpl}
	Let $\rsf J\subseteq\rsf H\subseteq\rsf I$. Then
	$$
	\rsf J\to\rsf H\subseteq\rsf J\to\rsf I. 
	$$
\end{lem}
\begin{proof}
	If $h\in\rsf H$ and $h\join g=\one$ for all $g\in\rsf J$ then
	$h\in\rsf J\to\rsf I$. 
\end{proof}

\begin{cor}\label{cor:inclImpl}
	Let $\rsf J\subseteq\rsf H\subseteq\rsf I$ and
	$\rsf J\to\rsf I\subseteq\rsf H$. Then
	$$
	\rsf J\to\rsf H=\rsf J\to\rsf I. 
	$$
\end{cor}
\begin{proof}
	LHS$\subseteq$RHS by the lemma. Conversely if $h\in\rsf J\to\rsf I$ 
	then 
	$h\in\rsf H$ has the defining property for $\rsf J\to\rsf H$ and so 
	is in $\rsf J\to\rsf H$. 
\end{proof}

\begin{cor}\label{cor:conJoint}
	$$
	\rsf J\to(\rsf J\join(\rsf J\to\rsf I))=\rsf J\to\rsf I. 
	$$
\end{cor}

\begin{lem}\label{lem:inclAgain}
	Let $\rsf J\subseteq\rsf H\subseteq\rsf I$. Then
	$$
	\rsf H\to\rsf I\subseteq\rsf J\to\rsf I. 
	$$
\end{lem}
\begin{proof}
	This is clear as $k\join h=\one$ for all $h\in\rsf H$ implies
	$k\join g=\one$ for all $g\in\rsf J$. 
\end{proof}

\subsection{Delta on Filters}

Now the critical lemma in defining our new $\Delta$ operation. 

\begin{lem}\label{lem:deltaOne}
	$(\rsf J\to\rsf I)\cup\Delta(\one, \rsf J)$ is compatible. 
\end{lem}
\begin{proof}
	If $x\in \rsf J\to\rsf I$ and $y\in\Delta(\one, \rsf J)$ then 
	$\Delta(\one,  y)\in\rsf J$ and so $x\join \Delta(\one,  y)=\one$.
\end{proof}

\begin{defn}\label{def:Delta}
	Let $\rsf J\subseteq\rsf I$. Then 
	$$
	\Delta(\rsf J, \rsf I)=\Delta(\one, \rsf J\to\rsf I)\join\rsf J. 
	$$
\end{defn}

The simplest special algebras in $\rsf I$ are the principal ones. In this case 
we obtain the following result. 

\begin{lem}\label{lem:DeltagOne}
	Let $g\in\rsf I$. Then $\Delta([g, \one], \rsf I)=\rsf I_{g}$. 
\end{lem}
\begin{proof}
	From \lemref{lem:implGG} we have $[g, \one]\to\rsf I=g\to\rsf I$ and 
	we know from \corref{cor:implSS} that
	$\Delta(\one, \rsf I_{g})\cap\rsf I=g\to\rsf I$. Thus
	$\Delta(\one, g\to\rsf I)\subseteq\rsf I_{g}$. 
	Also $g\in\rsf I_{g}$ so we have 
	$\Delta([g, \one], \rsf I)\subseteq\rsf I_{g}$. 
	
	Conversely, if $f\in\rsf I$ then 
	$\Delta(g\join f, f)=(g\join f)\meet\Delta(\one, g\to f)$ is in 
	$\Delta(1, g\to\rsf I)\join[g, \one]=\Delta([g, \one], \rsf I)$. 
\end{proof}

\begin{cor}\label{cor:DeltaDoublePrinc}
	Let $g\geq h$ is $\rsf I$. Then 
	$$
		\Delta([g, \one], [h, \one])=[\Delta(g, h), \one]. 	
	$$
\end{cor}
\begin{proof}
	As $\Delta([g, \one], [h, \one])= [h, \one]_{g}= [\Delta(g, h), \one]$. 
\end{proof}

For further properties of the $\Delta$ operation we need some 
facts about the interaction between $\to$ and $\Delta$. Here is the first. 

\begin{lem}\label{lem:implInDelta}
	$$
	\rsf J\to\Delta(\rsf J, \rsf I)=\Delta(\one, \rsf J\to\rsf I). 
	$$
\end{lem}
\begin{proof}
	Let $k\in\rsf J\to\rsf I$, $h=\Delta(\one, k)$ and $g\in\rsf J$. Then
	$k, g\in\rsf I$ implies they are compatible and so $\Delta(\one, k)\join g= h\join g=\one$.
	Thus $h\in \rsf J\to\Delta(\rsf J, \rsf I)$ and we get 
	$\Delta(\one, \rsf J\to\rsf I)\subseteq \rsf J\to\Delta(\rsf J, \rsf I)$.
	
	Conversely, suppose that $h\in\Delta(\rsf J, \rsf I)$ and for all $g\in\rsf J$ 
	we have $h\join g=\one$. Then
	there is some $k\in\rsf J\to\rsf H$ and $g'\in\rsf J$ such that
	$h=\Delta(\one, k)\meet g'$. Therefore 
	$\one = h\join g'= (\Delta(\one, k)\meet g')\join g'= g'$ and so
	$h=\Delta(\one, k)\in \Delta(\one, \rsf J\to\rsf I)$. 
\end{proof}

\begin{cor}\label{cor:doubleDelta}
	$$
	\Delta(\rsf J, \Delta(\rsf J, \rsf I))=\rsf J\join(\rsf J\to\rsf I). 
	$$
\end{cor}
\begin{proof}
	\begin{align*}
		\Delta(\rsf J, \Delta(\rsf J, \rsf I)) & =\Delta(\one, 
		\rsf J\to\Delta(\rsf J, \rsf I))\join\rsf J  \\
		 & =\Delta(\one, \Delta(\one, \rsf J\to\rsf I))\join\rsf J  \\
		 & =(\rsf J\to\rsf I)\join\rsf J. 
	\end{align*}
\end{proof}

\begin{lem}\label{lem:inclDelta}
	Let $\rsf J\subseteq\rsf H\subseteq\rsf I$. Then
	$$
	\Delta(\rsf J, \rsf H)\subseteq\Delta(\rsf J, \rsf I). 
	$$
\end{lem}
\begin{proof}
	As $\Delta(\one, \rsf J\to\rsf H)\join\rsf J\subseteq
	\Delta(\one, \rsf J\to\rsf I)\join\rsf J$. 
\end{proof}

\begin{lem}\label{lem:interDelta}
	$\rsf I\cap\Delta(\rsf J, \rsf I)=\rsf J$. 
\end{lem}
\begin{proof}
	Clearly $\rsf J\subseteq \rsf I\cap\Delta(\rsf J, \rsf I)$. 
	
	Let $g\in\rsf J$ and $k\in\rsf J\to\rsf I$ be such that
	$f=g\meet\Delta(\one, k)\in\rsf I\cap\Delta(\rsf J, \rsf I)$. Then $k\in\rsf I$ so
	$k$ and $\Delta(\one, k)$ are compatible. Thus $k=\one$ and so $f=g\in\rsf J$. 
\end{proof}

\subsection{Boolean elements}
Corollary \ref{cor:doubleDelta} shows us what happens to $\Delta(\rsf  Q, 
\Delta(\rsf  Q, \rsf  P))$. We are interested in knowing when this produces 
$\rsf  P$. 

\begin{defn}\label{def:Boolean}
	Let $\rsf  P$ and $\rsf Q$ be special subalgebras. Then
	\begin{enumerate}[(a)]
		\item $\rsf  Q$ is \emph{weakly $\rsf  P$-Boolean} iff $\rsf  Q\subseteq\rsf 
		P$ and
		$(\rsf  Q\to\rsf  P)\to\rsf  P=\rsf  Q$. 
		
		\item $\rsf  Q$ is \emph{$\rsf  P$-Boolean} iff $\rsf  Q\subseteq\rsf 
		P$ and $\rsf  Q\join(\rsf  Q\to\rsf  P)=\rsf  P$. 
	\end{enumerate}
\end{defn}

Before continuing however we show that ``weak'' really is weaker. 

\begin{lem}\label{lem:BooleanANotherWay}
	Suppose that $\rsf  Q$ is $\rsf  P$-Boolean. 
	Then $\rsf  Q$ is weakly $\rsf  P$-Boolean. 
\end{lem}
\begin{proof}
	We know that $\rsf  Q\subseteq(\rsf  Q\to\rsf  P)\to\rsf  P$. 
	
	Since $\rsf  Q\join(\rsf  Q\to\rsf  P)=\rsf  P$ we also have that
	$(\rsf  Q\to\rsf  P)\supset\rsf  P\subseteq\rsf  Q$. 
\end{proof}

And now the simplest examples of $\rsf  P$-Boolean subalgebras. 

\begin{lem}\label{lem:implgTwice}
	Let $g\in\rsf  P$. Then 
	$[g, \one]$ is $\rsf  P$-Boolean. 
\end{lem}
\begin{proof}
	We know that 
	$$
	\Delta([g, \one], \rsf  P_{g})=(\rsf  P_{g})_{g}=\rsf  P
	$$
	and so 
	\begin{align*}
		\rsf  P & =[g, \one]\join\Delta(\one, [g, \one]\to\rsf  P_{g})  \\
		 & =[g, \one]\join \Delta(\one, [g, \one]\to\Delta([g, \one], \rsf  P))  \\
		 & =[g, \one]\join\Delta(\one, \Delta(\one, [g, \one]\to\rsf  P))\\
		 & = [g, \one]\join (g\to\rsf  P). 
	\end{align*}
\end{proof}

Essentially because we have so many internal automorphisms we can show 
that Boolean is not a local concept -- that is if $\rsf  Q$ is $\rsf P$-Boolean 
somewhere then it is Boolean in all special subalgebras equivalent to $\rsf P$. And similarly for weakly 
Boolean. 

\begin{lem}\label{lem:moving}
	Let $\rsf  P\sim\rsf H$ and $\rsf  Q\subseteq\rsf  P\cap\rsf H$ be special subalgebras. 
	Let $\beta=\beta_{\rsf  P\rsf H}$ (and so
		$\beta^{-1}=\beta_{\rsf H\rsf  P}$). 
	Then $\beta[\rsf  Q\to\rsf  P]=\beta[\rsf  Q]\to\rsf H$. 
\end{lem}
\begin{proof}
	Indeed if $g\in\rsf  Q$ and $h\in\rsf  Q\to\rsf  P$ then we have 
		\begin{align*}
			\one & =\beta(h\join g)  \\
			 & =\beta(h)\join\beta(g)
		\end{align*}
		and so $\beta(h)\in\beta[\rsf  Q]\to\rsf H$. 
		
		Likewise, if $h\in\beta[\rsf  Q]\to\rsf H$ and $g\in\rsf  Q$ then
		$\one= h\join\beta(g)= \beta(\beta^{-1}(h)\join g)$ so that
		$\beta^{-1}(h)\join g=\one$. Thus $\beta^{-1}(h)\in\rsf  Q\to\rsf  P$
		whence $h=\beta(\beta^{-1}(h))\in \beta[\rsf  Q\to\rsf  P]$. 
\end{proof}

\begin{thm}\label{thm:Boolean}
	Let $\rsf  Q$ be $\rsf  P$-Boolean, and $\rsf P\sim\rsf R$ with $\rsf Q\subseteq\rsf R$. 
	Then $\rsf  Q$ is $\rsf R$-Boolean. 
\end{thm}
\begin{proof}
	We have $\rsf  Q\join(\rsf  Q\to\rsf  P)=\rsf  P$ and $\rsf 
	Q\subseteq\rsf R$. Let $\beta=\beta_{\rsf P\rsf R}$, $h\in\rsf R$ and find 
	$g\in\rsf  Q$, $k\in\rsf  Q\to\rsf  P$ with $\beta^{-1}(h)=g\meet k$. 
	Then
	$h= \beta(\beta^{-1}(h))= \beta(g\meet k)= \beta(g)\meet\beta(k)= 
	 g\meet\beta(k)$ as $g\in\rsf H$ implies $\beta(g)=g$. As 
	 $\beta(k)\in \beta[\rsf  Q\to\rsf  P]= \beta[\rsf  Q]\to\rsf H= 
	 \rsf  Q\to\rsf H$ we have $h\in \rsf  Q\join(\rsf  Q\to\rsf H)$. 
\end{proof}

\begin{thm}\label{thm:wkBoolean}
	Let $\rsf  Q$ be weakly $\rsf  P$-Boolean for some special subalgebra $\rsf  P$, 
	and $\rsf P\sim\rsf R$ with $\rsf Q\subseteq\rsf R$. 
	Then $\rsf  Q$ is weakly $\rsf R$-Boolean. 
\end{thm}
\begin{proof}
	\begin{enumerate}[{Claim }1:]
		\item $\beta[\rsf  Q]=\rsf  Q$ -- since $\rsf  Q\subseteq\rsf R$ 
		implies $\beta\restrict\rsf  Q$ is the identity. 
	
		\item Now suppose that $\rsf  Q$ is weakly $\rsf  P$-Boolean. Then
		\begin{align*}
			\rsf  Q & =\beta[\rsf  Q]  \\
			 & =\beta[(\rsf  Q\to\rsf  P)\to\rsf  P]  \\
			 & =\beta[\rsf  Q\to\rsf  P]\to\rsf R  \\
			 & =(\beta[\rsf  Q]\to\rsf R)\to\rsf R\\
			 & = (\rsf  Q\to\rsf R)\to\rsf R. 
		\end{align*}
	\end{enumerate}	
\end{proof}

We need to know certain persistence properties of Boolean-ness. 
\begin{lem}\label{lem:upwards}
	Let $\rsf  Q\subseteq\rsf  R\subseteq\rsf  P$ be $\rsf  P$-Boolean. 
	Then $\rsf  Q$ is $\rsf  R$-Boolean and $\rsf  Q\to\rsf  R=(\rsf  Q\to\rsf 
	F)\cap\rsf  R$. 
\end{lem}
\begin{proof}
	First we note that $\rsf  Q\to\rsf  R=(\rsf  Q\to\rsf  P)\cap\rsf  R$ as 
	$x\in$LHS iff $x\in\rsf  R$ and for all $g\in\rsf  Q$ $x\join 
	g=\one$ iff $x\in$RHS. 
	
	Thus we have 
	\begin{align*}
		\rsf  R & =\rsf  P\cap\rsf  R  \\
		 & =(\rsf  Q\join(\rsf  Q\to\rsf  P))\cap\rsf  R  \\
		 & = (\rsf  Q\cap\rsf  R)\join((\rsf  Q\to\rsf  P)\cap\rsf  R) \\
		 & =\rsf  Q\join(\rsf  Q\to\rsf  R). 
	\end{align*}
\end{proof}

\begin{lem}\label{lem:middle}
	Let $\rsf  Q$ be $\rsf  R$-Boolean, $\rsf  R$ be $\rsf  P$-Boolean. Then 
	$\rsf  Q$ is $\rsf  P$-Boolean. 
\end{lem}
\begin{proof}
	Let $f\in\rsf  P$. Then there is some $h\in\rsf  R$ and $k\in\rsf 
	H\to\rsf  P$ such that $h\meet k=f$. Also there is some 
	$g\in\rsf  Q$ and $l\in\rsf  Q\to\rsf  R$ such that $h=g\meet l$. 
	Thus $g\meet l\meet k=f$ -- so it suffices to show that $l\meet 
	k\in\rsf  Q\to\rsf  P$. 
	
	Clearly $k\meet l\in\rsf  P$. So let $p\in\rsf  Q$. Then
	$\rsf  Q\subseteq\rsf  R$ and $k\in\rsf  R\to\rsf  P$ implies $p\join 
	 k=\one$. $l\in\rsf  Q\to\rsf  R$ implies $p\join l=\one$. 
	 Therefore
	 $p\join(k\meet l)=(p\join k)\meet(p\join l)= \one\meet\one= \one$. 
\end{proof}

So far we have few examples of Boolean special subalgebras. The next lemma 
produces many more. 

\begin{lem}\label{lem:lots}
	Let $\rsf  P\sim\rsf  R$. Then $\rsf  P\cap\rsf  R$ is $\rsf  P$-Boolean 
	and
	$$
	(\rsf  P\cap\rsf  R)\to\rsf  P=\Delta(\one, \rsf  R)\cap\rsf  P. 
	$$
\end{lem}
\begin{proof}
	First we show that $(\rsf  P\cap\rsf  R)\to\rsf  P=\Delta(\one, \rsf  R)\cap\rsf  P$. 
	
	Let $f\in\rsf  P\cap\rsf  R$ and $k\in\Delta(\one, \rsf  R)\cap\rsf  P$. 
	Then $\Delta(\one, k)\in\rsf  R$ so $\Delta(\one,  k)$ and $f$ are compatible, ie $k\join f=\one$. Hence 
	$\Delta(\one, \rsf  R)\cap\rsf  P\subseteq (\rsf  P\cap\rsf  R)\to\rsf  P$. 
	
	Conversely suppose that $k\in (\rsf  P\cap\rsf  R)\to\rsf  P$. Let
	$h\in\rsf  R$. Then $h\join k\in\rsf  P\cap\rsf  R$ and so
	$h\join k= (h\join k)\join k= \one$. As there is some $k'\sim k$ 
	in $\rsf  R$ this implies $k'\join k=\one$ and (as $k\sim k'$) we have 
	$k=\Delta(\one, k')$. Thus
	$k\in \Delta(\one, \rsf R)\cap\rsf  P$. 
	
	Now let $f\in\rsf  P$. Then let $f'\in\rsf  R$ with $f'\sim f$. 
	Then 
	$(f\join f')\to f= f\meet \Delta(\one, \Delta(f'\join f, f))=
	f\meet \Delta(\one, f')\in \rsf  P\cap\Delta(\one, \rsf  R)$. 
	Also 
	$f\join f'\in\rsf  P\cap\rsf  R$ and $(f\join f')\meet ((f\join f')\to 
	f)=f$ so $f\in (\rsf  P\cap\rsf  R)\join(\rsf  P\cap\Delta(\one, \rsf  R))$. 
\end{proof}

\begin{cor}\label{cor:lots}
	Let $\rsf  P\sim\rsf  R$. Then 
	$$
	\Delta(\rsf  P\cap\rsf  R, \rsf  P)=\rsf  R. 
	$$
\end{cor}
\begin{proof}
	\begin{align*}
		\Delta(\rsf  P\cap\rsf  R, \rsf  P)&=(\rsf  P\cap\rsf  R)\join
		\Delta(\one, (\rsf  P\cap\rsf  R)\to\rsf  P)\\
		&=(\rsf  P\cap\rsf  R)\join\Delta(\one, \Delta(\one, \rsf  R)\cap\rsf  P)\\
		&=(\rsf  P\cap\rsf  R)\join(\rsf  R\cap\Delta(\one, \rsf  P))\\
		&=(\rsf  P\cap\rsf  R)\join((\rsf  P\cap\rsf  R)\to\rsf  R)\\
		&=\rsf  R
	\end{align*}
	since $\rsf  P\cap\rsf  R$ is also $\rsf  R$-Boolean. 
\end{proof}

\begin{lem}\label{lem:DeltaInMR}
	Let $g, h$ in $\mathcal L$ be such that $g\meet h$ exists and 
	$g\join h=\one$. Then $\Delta(g, g\meet h)=g\meet\Delta(\one, h)$. 
\end{lem}
\begin{proof}
	\begin{align*}
		\Delta(g, g\meet h) & =g\meet\Delta(\one, g\to(g\meet h))  \\
		 & =g\meet \Delta(\one, (g\join h)\to h)  && \text{ by modularity in 
		 }[g\meet h, \one]\\ 
		 & = g\meet \Delta(\one, \one\to h) \\
		 & = g\meet \Delta(\one, h)
	\end{align*}
\end{proof}

\begin{thm}\label{thm:lots}
	$\rsf  R\sim\rsf  P$ iff there is an $\rsf  P$-Boolean subalgebra $\rsf  Q$ 
	such that $\rsf  R=\Delta(\rsf  Q, \rsf  P)$. 
\end{thm}
\begin{proof}
	The right to left direction is the last corollary. 
	
	So we want to prove that $\Delta(\rsf  Q, \rsf  P)\sim\rsf  P$ 
	whenever $\rsf  Q$ is $\rsf  P$-Boolean. 
	
	Let $f\in\rsf  P$. We will show that there is some $f'\in\Delta(\rsf  Q, 
	\rsf  P)$ with $f\sim f'$. 
	As $\rsf  Q\join(\rsf  Q\to\rsf  P)=\rsf  P$ we can find $g\in\rsf  Q$ and
	$h\in\rsf  Q\to\rsf  P$ with $f=g\meet h$. As $g\join h=\one$ we know 
	that
	$\Delta(g, g\meet h)=g\meet\Delta(\one, h)$. But
	$g\meet\Delta(\one, h)\in \rsf  Q\join\Delta(\one, \rsf  Q\to\rsf  R)= 
	\Delta(\rsf  Q, \rsf  P)$ and $f=g\meet h\sim \Delta(g, g\meet h)= g\meet\Delta(\one, h)$.
\end{proof}

The Boolean elements have nice properties with respect to $\Delta$. We 
want to show more -- that the set of $\rsf  P$-Boolean elements is a 
Boolean subalgebra of $[\rsf  P, \Set{\one}]$ with the reverse order.

It suffices to show closure under $\cap$ and $\join$ -- closure under 
$\to$ follows from \lemref{lem:BooleanANotherWay}. 

\begin{lem}\label{lem:interStrBool}
	Let $\rsf  Q_{1}$ and $\rsf  Q_{2}$ be $\rsf  P$-Boolean. Then 
	$(\rsf  Q_{1}\to\rsf  P)\join(\rsf  Q_{2}\to\rsf  P)= (\rsf  Q_{1}\cap\rsf 
	G_{2})\to\rsf  P$. 
\end{lem}
\begin{proof}
	Suppose that $h_{i}\in\rsf  Q_{i}\to\rsf  P$ and $g\in \rsf 
	Q_{1}\cap\rsf  Q_{2}$. Then
	$(h_{1}\meet h_{2})\join g= (h_{1}\join g)\meet (h_{2}\join g)= 
	\one\meet\one=\one$ and so 
	$(h_{1}\meet h_{2}\in (\rsf  Q_{1}\cap\rsf  Q_{2})\to\rsf  P$. 
	
	Conversely, let $h\join g=\one$ for all $g\in\rsf  Q_{1}\cap\rsf  Q_{2}$. 
	As $\rsf  Q_{i}$ are both $\rsf  P$-Boolean there exists $h_{i}\in\rsf 
	G_{i}\to\rsf  P$ and $g_{i}\in\rsf  Q_{i}$ with
	$h= h_{1}\meet g_{1}= h_{2}\meet g_{2}$. Then
	\begin{align*}
		h_{1}\meet h_{2}\meet(g_{1}\join g_{2})&=
		(h_{1}\meet h_{2}\meet g_{1})\join(h_{1}\meet h_{2}\meet g_{2})\\
		&=(h_{2}\meet h)\join(h_{1\meet h})\\
		&=h\meet h=h. 
	\end{align*}
	As $h_{1}\meet h_{2}\in (\rsf  Q_{1}\to\rsf  P)\join(\rsf  Q_{2}\to\rsf  P)$
	and $g_{1}\join g_{2}\in \rsf  Q_{1}\cap\rsf  Q_{2}$
	we then have 
	$h= [h\join (h_{1}\meet h_{2})]\meet (h\join g_{1}\join g_{2})=
	h\join (h_{1}\meet h_{2})$ and so $h=h_{1}\meet h_{2}$ is in 
	$(\rsf  Q_{1}\to\rsf  P)\join(\rsf  Q_{2}\to\rsf  P)$.	
\end{proof}

\begin{cor}\label{cor:interStrBool}
	Let $\rsf  Q_{1}$ and $\rsf  Q_{2}$ be $\rsf  P$-Boolean. Then so is
	$\rsf  Q_{1}\cap\rsf	G_{2}$. 
\end{cor}
\begin{proof}
	Let $f\in\rsf  P$. 
	As $\rsf  Q_{i}$ are both $\rsf  P$-Boolean there exists $h_{i}\in\rsf 
	G_{i}\to\rsf  P$ and $g_{i}\in\rsf  Q_{i}$ with
	$f= h_{1}\meet g_{1}= h_{2}\meet g_{2}$. Then as above
	$f= h_{1}\meet h_{2}\meet(g_{1}\join g_{2})$ and
	$g_{1}\join g_{2}\in\rsf  Q_{1}\cap\rsf  Q_{2}$ and 
	$h_{1}\meet h_{2}\in (\rsf  Q_{1}\to\rsf  P)\join(\rsf  Q_{2}\to\rsf  P)= (\rsf  Q_{1}\cap\rsf 
	G_{2})\to\rsf  P$.
\end{proof}

\begin{cor}\label{cor:joinStrBool}
	Let $\rsf  Q_{1}$ and $\rsf  Q_{2}$ be $\rsf  P$-Boolean. Then so is
	$\rsf  Q_{1}\join\rsf  Q_{2}$. 
\end{cor}
\begin{proof}
	Since we have $(\rsf  Q\to\rsf  P)\to\rsf  P=\rsf  Q$ for $\rsf 
	P$-Booleans we know that
	$\rsf  Q_{i}\to\rsf  P$ are also $\rsf  P$-Boolean and so
	\begin{align*}
		\rsf  Q_{1}\join\rsf  Q_{2} & = ((\rsf  Q_{1}\to\rsf  P)\to\rsf  P)\join 
		((\rsf  Q_{2}\to\rsf  P)\to\rsf  P)\\
		 & = ((\rsf  Q_{1}\to\rsf  P)\cap
				(\rsf  Q_{2}\to\rsf  P))\to\rsf  P \\
				\intertext{Therefore}
		(\rsf  Q_{1}\join\rsf  Q_{2})\to\rsf  P & =
		(((\rsf  Q_{1}\to\rsf  P)\cap
				(\rsf  Q_{2}\to\rsf  P))\to\rsf  P)\to\rsf  P\\
		 & =(\rsf  Q_{1}\to\rsf  P)\cap
				(\rsf  Q_{2}\to\rsf  P). 
	\end{align*}
	Thus we have 
	$$
	(\rsf  Q_{1}\join\rsf  Q_{2})\join((\rsf  Q_{1}\join\rsf  Q_{2})\to\rsf  P)=
	(\rsf  Q_{1}\join\rsf  Q_{2})\join((\rsf  Q_{1}\to\rsf  P)\cap
				(\rsf  Q_{2}\to\rsf  P)). 
	$$
	Let $f\in\rsf  P$ and $g_{i}\in\rsf  Q_{i}$, $h_{i}\in\rsf  Q_{i}\to\rsf 
	P$ be such that $f=g_{i}\meet h_{i}$. Then $f\le g_{1}, g_{2}$ so that 
	$g_{1}\meet g_{2}\in\rsf  Q_{1}\join\rsf  Q_{2}$, 
	$h_{1}\join h_{2}\in (\rsf  Q_{1}\to\rsf  P)\cap
				(\rsf  Q_{2}\to\rsf  P)$ and 
	\begin{align*}
		g_{1}\meet g_{2}\meet (h_{1}\join h_{2}) & =
		(g_{1}\meet g_{2}\meet h_{1})\join(g_{1}\meet g_{2}\meet h_{2})\\
		&=(g_{2}\meet f)\join(g_{1}\meet f)\\
		&=f\meet f&&\text{ as }f\le g_{i}\\
		&=f. 
	\end{align*}
\end{proof}

Thus we have 
\begin{thm}\label{thm:BooleanAlgebra}
	Let $\rsf  P$ be any special subalgebra. Then
	$\Set{\rsf  Q | \rsf  Q\text{ is }\rsf  P\text{-Boolean}}$ ordered by reverse 
	inclusion is a Boolean algebra with $\meet=\join$, $\join=\cap$, 
	$1=\Set{\one}$, $0=\rsf  P$ and $\comp{\rsf  Q}=\rsf  Q\to\rsf  P$. 
\end{thm}
\begin{proof}
	This is immediate from lemma \ref{lem:distrib} and preceding remarks, 
	and from \lemref{lem:BooleanANotherWay}.  
\end{proof}

We need a stronger closure property for Boolean filters under 
intersection. 
\begin{lem}\label{lem:closIntersect}
	Let $\rsf  P\sim\rsf  R$, $\rsf  Q$ be $\rsf  P$-Boolean and $\rsf K$ 
	be $\rsf  R$-Boolean. Then $\rsf  Q\cap\rsf K$ is $\rsf  P\cap\rsf 
	R$-Boolean. 
\end{lem}
\begin{proof}
	Let $p\in\rsf  P\cap\rsf  R$ be arbitrary. Choose
	$g\in\rsf  Q$, $g'\in\rsf  Q\to\rsf  P$ with $g\meet g'=p$ and choose
	$k\in\rsf K$, $k'\in\rsf K\to\rsf  R$ with $k\meet k'=p$. 
	
	Then $g'$ and $k'$ are both above $p$ so $g'\meet k'$ exists and is 
	is $\rsf  P\cap\rsf  R$. Also
	$(g\join k)\meet(g'\meet k')=p$. $g\join k\in\rsf  Q\cap\rsf K$ so we 
	need to show that $g'\meet k'$ is in $(\rsf  Q\cap\rsf K)\to(\rsf 
	P\cap\rsf  R)$. Let 
	$q\in\rsf  Q\cap\rsf K$. Then $q\join g'=\one=q\join k'$ so that
	$q\join (g'\meet k')= (q\join g')\meet(q\join k')= \one$. 
\end{proof}

\begin{cor}\label{cor:closIntersect}
	Let $\rsf  Q$ be $\rsf P$-Boolean, $\rsf K$ be $\rsf R$-Boolean and 
	$\rsf P\sim\rsf R$. Then $\rsf  Q\cap\rsf K$ is $\rsf P$-Boolean. 
\end{cor}
\begin{proof}
	The lemma tells us that $\rsf  Q\cap \rsf K$ is $\rsf  P\cap\rsf  R$-Boolean. 
	Theorem \ref{thm:lots} tells us that $\rsf  P\cap\rsf  R$ is $\rsf  P$-Boolean.
	And from \lemref{lem:middle} we have $\rsf  Q\cap\rsf K$ to be $\rsf 
	P$-Boolean. 
\end{proof}

The last closure property we need is with respect to $\Delta$. 
\begin{lem}\label{lem:deltaGivesStrA}
	Let $\rsf  Q\subseteq\rsf  R\subseteq\rsf  P$ be $\rsf  P$-Boolean 
	subalgebras. Then
	$$
		\Delta(\rsf  Q, \rsf  R)\to\Delta(\rsf  Q, \rsf  P)=\Delta(\one, \rsf 
		R\to\rsf  P). 
	$$
\end{lem}
\begin{proof}
	As $\rsf  Q\subseteq\rsf  R\subseteq\rsf  P$ in a Boolean algebra we have
	$$
	(\rsf  Q\to\rsf  R)\to(\rsf  Q\to\rsf  P)=\rsf  R\to\rsf  P. 
	$$
	Also we have 
	\begin{align*}
		\Delta(\rsf  Q, \rsf  R) & =\rsf  Q\join\Delta(\one, \rsf  Q\to\rsf  R)  \\
		\Delta(\rsf  Q, \rsf  P) & =\rsf  Q\join\Delta(\one, \rsf  Q\to\rsf  P). 
	\end{align*}
	Let $x\in\Delta(\rsf  Q, \rsf  R)$ and $g\in \rsf  Q$, $h\in\rsf  Q\to\rsf  R $
	with $x=g\meet\Delta(\one, h)$. 
	Let $y\in\Delta(\rsf  Q, \rsf  P)$ and $g'\in \rsf  Q$, $f\in\rsf  Q\to\rsf 
	P $ with $y=g'\meet\Delta(\one, f)$ and suppose that $x\join y=\one$ 
	for all such $x$. 
	Then 
	\begin{align*}
		y\join x & = (g'\meet\Delta(\one, f))\join
		(g\meet\Delta(\one, h))\\
		 & =(g'\join g)\meet(g'\join\Delta(\one, h))\meet
		 (\Delta(\one, f)\join g)\meet\Delta(\one, f\join h)\\
		 & =(g'\join g)\meet\Delta(\one, f\join h) 
	\end{align*}
	since $g$ and $f$ are compatible, as are $g'$ and $h$.
	
	Thus $g'\join g=\one$ and $f\join h=\one$ for all $g\in\rsf  Q$ and 
	all $h\in\rsf  Q\to\rsf  R$. Choosing $g=g'$ implies $g'=\one$
	and so
	$f\in (\rsf  Q\to\rsf  R)\to(\rsf  Q\to\rsf  P)= \rsf  R\to\rsf  P$. 
	Hence $y=\Delta(\one, f)\in\Delta(\one, \rsf  R\to\rsf  P)$. 
	
	Conversely if $f\in\rsf  R\to\rsf  P$ then $g\join \Delta(\one, 
	f)=\one$ for all $g\in\rsf  Q$. And 
	$f\in (\rsf  Q\to\rsf  R)\to(\rsf  Q\to\rsf  P)$ implies $h\join f=\one$ 
	for all
	$h\in\rsf  Q\to\rsf  R$. Hence $(g\meet\Delta(\one, h))\join\Delta(\one, 
	f)=\one$ and so 
	$\Delta(\one, f)$ is in $\Delta(\rsf  Q, \rsf  R)\to\Delta(\rsf  Q, \rsf  P)$. 
\end{proof}

\begin{lem}\label{lem:deltaGivesStr}
	Let $\rsf  Q\subseteq\rsf  R\subseteq\rsf  P$ be $\rsf  P$-Boolean 
	subalgebras. Then
	$\Delta(\rsf  Q, \rsf  R)$ is $\Delta(\rsf  Q, \rsf  P)$-Boolean. 
\end{lem}
\begin{proof}
	Since 
	\begin{align*}
		\Delta(\rsf  Q, \rsf  R)\join(\Delta(\rsf  Q, \rsf  R)\to\Delta(\rsf  Q, \rsf  P)) & =
		\rsf  Q\join\Delta(\one, \rsf  Q\to\rsf  R)\join
		\Delta(\one, \rsf  R\to\rsf  P)\\
		 & =
		\rsf  Q\join\Delta(\one, \rsf  Q\to\rsf  R)\join
		\Delta(\one, (\rsf  Q\to\rsf  R)\to(\rsf  Q\to\rsf  P))  \\
		 & = \rsf  Q\join\Delta((\one, \rsf  Q\to\rsf  R)\join
		((\rsf  Q\to\rsf  R)\to(\rsf  Q\to\rsf  P)))  \\
		 & =\rsf  Q\join\Delta(\one, \rsf  Q\to\rsf  P)\\
		 &= \Delta(\rsf  Q, \rsf  P). 
	\end{align*}
\end{proof}

From this lemma  we can derive another property of $\Delta$. 
\begin{lem}\label{lem:implDDD}
	Let $\rsf  Q\subseteq\rsf  R\subseteq\rsf  P$ be $\rsf  P$-Boolean 
	subalgebras. Then
	$$
		\rsf  Q\to\Delta(\rsf  R, \rsf  P)=(\rsf  Q\to\rsf  R)\join\Delta(\one, 
		\rsf  R\to\rsf  P). 
	$$
\end{lem}
\begin{proof}
	The RHS is clearly a subset of $\Delta(\rsf  R, \rsf  P)$. 
	Let $g\in\rsf  Q$. If $h\in\rsf  Q\to\rsf  R$ then $h\join g=\one$. 
	If $k\in \Delta(\one, \rsf  R\to\rsf  P)$ then $\Delta(\one, k)\in\rsf R\to\rsf P\subseteq\rsf Q\to\rsf P$ 
	so that $g\join k=\one$. Thus the RHS is a subset of the LHS. 
	
	Conversely suppose that $h=h_{1}\meet h_{2}$ is in $\rsf  R\join\Delta(\one, 
	\rsf  R\to\rsf  P)= \Delta(\rsf  R, \rsf  P)$ and $g\join h=\one$
	for all $g\in\rsf  Q$. Then $g\join h_{1}=\one$ for all $g\in\rsf  Q$ 
	and so $h_{1}\in\rsf  Q\to\rsf  R$. Thus the LHS is a subset of the RHS. 
\end{proof}

\begin{cor}\label{cor:iteratedDelta}
	Let $\rsf  Q\subseteq\rsf  R\subseteq\rsf  P$ be $\rsf  P$-Boolean 
	subalgebras. Then
	$$
		\Delta(\rsf  Q, \Delta(\rsf  R, \rsf  P))=\Delta(\Delta(\rsf  Q, \rsf  R), 
		\Delta(\rsf  Q, \rsf  P)). 
	$$
\end{cor}
\begin{proof}
	\begin{align*}
		\Delta(\rsf  Q, \Delta(\rsf  R, \rsf  P)) & =
		\rsf  Q\join\Delta(\one, \rsf  Q\to\Delta(\rsf  R, \rsf  P))\\
		 & =\rsf  Q\join\Delta(\one, (\rsf  Q\to\rsf  R)\join\Delta(\one, 
		\rsf  R\to\rsf  P))  \\
		 & =\rsf  Q\join \Delta(\one, \rsf  Q\to\rsf  R)\join(\rsf  R\to\rsf  P)  \\
		 & = \Delta(\rsf  Q, \rsf  R)\join\Delta(\one, \Delta(\rsf  Q, \rsf  R)\to
		 \Delta(\rsf  Q, \rsf  P)) \\
		 & = \Delta(\Delta(\rsf  Q, \rsf  R), \Delta(\rsf  Q, \rsf  P)). 
	\end{align*}
\end{proof}

\subsection{An MR-algebra}
The results of the last section show us that there is a natural MR-algebra sitting 
over the top of any cubic algebra. The first theorem describes the case for 
cubic algebras with g-covers.

\begin{thm}\label{thm:MRalgFilter}
    Let $\mathcal L$ be a cubic algebra with a g-cover. 
	Let $\mathcal L_{sB}$ be the set of all special subalgebras that are $\rsf P$-Boolean for some
	g-cover $\rsf P$. Order these by reverse inclusion. 
	Then 
	\begin{enumerate}[(a)]
		\item $\mathcal L_{sB}$ contains $\Set{\one}$ and is closed
		 under the operations $\join$ and $\Delta$. 
		
		\item $\brk<\mathcal L_{sB}, \Set{\one}, \join, \Delta>$ is an atomic MR-algebra. 
	
		\item The mapping $e\colon\mathcal L\to \mathcal L_{sB}$ given by 
		$g\mapsto[g, \one]$ is a full embedding. 
	
		\item The atoms of $\mathcal L_{sB}$ are exactly the g-covers of 
		$\mathcal L$. 
	\end{enumerate}
\end{thm}
\begin{proof}
	\begin{enumerate}[(a)]
		\item It is easy to see that $\one\to\rsf =\rsf P$ for all filters 
		$\rsf P$. Corollary \ref{cor:closIntersect} and \lemref{lem:deltaGivesStr}
		give the closure under join and Delta respectively. 
		
		\item 
		We will proceed sequentially through the axioms. 
		\begin{enumerate}[i.]
		\item  if $x\le y$ then $\Delta(y, x)\join x = y$ -- this is 
		\lemref{lem:interDelta}. 
	
		\item  if $x\le y\le z$ then $\Delta(z, \Delta(y, x))=\Delta(\Delta(z, 
			y), \Delta(z, x))$ -- this is \corref{cor:iteratedDelta}. 
	
		\item  if $x\le y$ then $\Delta(y, \Delta(y, x))=x$ -- this 
		is \corref{cor:doubleDelta} and the definition of $\rsf 
		F$-Boolean. 
	
		\item  if $x\le y\le z$ then $\Delta(z, x)\le \Delta(z, 
		y)$ -- this is \lemref{lem:inclDelta}. 
	
		\item[] Let $xy=\Delta(1, \Delta(x\join y, y))\join y$ for any $x$, $y$ 
		in $\mathcal L$. 
			
			First we note that if $\rsf Q\subseteq\rsf P$ then
			\begin{align*}
				\Delta(\one, \Delta(\rsf Q, \rsf P))\cap\rsf P & =
				\Delta(\one, \rsf Q\join\Delta(\one, \rsf Q\to\rsf P))\cap\rsf P\\
				 & = (\Delta(\one, \rsf Q)\join(\rsf Q\to\rsf P))\cap\rsf P . 
			\end{align*}
			If $g\in\rsf Q$ and $h\in\rsf P$ is such that $\Delta(\one, g)\meet h\in\rsf P$ 
			then $g=\Delta(\one, g)$ (since $g\simeq\Delta(\one, g)$ and
			$g\meet \Delta(\one, g)$ exists). Thus 
			$(\Delta(\one, \rsf Q)\join(\rsf Q\to\rsf P))\cap\rsf P= 
			\rsf Q\to\rsf P$. 
	
		\item  $(xy)y=x\join y$ and 
		\item  $x(yz)=y(xz)$. 
			These last two properties hold as $\mathcal L_{sB}$ is locally 
			Boolean and hence an implication algebra. 
	\end{enumerate}
		
		To see that $\mathcal L_{sB}$ is an MR-algebra it suffices to note 
		that if $\rsf Q_{1}$ and $\rsf Q_{2}$ are in $\mathcal L_{sB}$ and we have 
		g-covers $\rsf P_{1}, \rsf P_{2}$ with $\rsf Q_{i}\subseteq\rsf 
		P_{i}$ then $\Delta(\rsf P_{1}\cap\rsf P_{2}, \rsf P_{2})= \rsf 
		P_{1}\supseteq\rsf Q_{1}$ so that $\rsf P_{2}\preccurlyeq\rsf Q_{1}$. 
		It is clear that $\rsf P_{2}\preccurlyeq\rsf Q_{2}$.
	
		\item It is clear that this mapping preserves 
		order and join. Preservation of $\Delta$ is \corref{cor:DeltaDoublePrinc}. 
		
		It is full because $[g, \one]\subseteq\rsf Q$ whenever $g\in\rsf 
		G$. 
	
		\item This is \thmref{thm:lots}. 
	\end{enumerate}
\end{proof}

The structure $\mathcal L_{sB}$ is another notion of envelope for 
cubic algebras. The existence of such an envelope -- it is an 
MR-algebra with a g-filter into which $\mathcal L$ embeds as a full 
subalgebra -- implies that $\mathcal L$ has a g-cover,  so this 
result cannot be directly extended to all cubic algebras. 

We note that if $\mathcal L$ is finite then $\mathcal L_{sB}$ is the 
same as the enveloping algebra given by \thmref{thm:envAlg}.


\begin{bibdiv}
\begin{biblist}
        \DefineName{cgb}{Bailey, Colin G.}
    \DefineName{jso}{Oliveira,  Joseph S.}

\bib{Abb:bk}{book}{
author={Abbott,J.C.}, 
title={Sets, Lattices, and Boolean Algebras}, 
publisher={Allyn and Bacon, Boston, MA}, 
date={1969}
}

\bib{BO:eq}{article}{
title={An Axiomatization for Cubic Algebras}, 
author={cgb}, 
author={jso}, 
book={
    title={Mathematical Essays in Honor of Gian-Carlo Rota}, 
    editor={Sagan,  Bruce E.}, 
    editor={Stanley, Richard P.}, 
    publisher={Birkha\"user}, 
    date={1998.}, 
}, 
pages={305--334}
}

\bib{BO:filAlg}{article}{
title={The Algebra of Filters of a Cubic Algebra}, 
author={cgb}, 
author={jso},
status={in preparation},
eprint={arXiv:0901.4933v1 [math.RA]} 
}

\bib{BO:cong}{article}{
title={Congruences and Homomorphisms of Cubic Algebras}, 
author={cgb}, 
author={jso},
status={in preparation} 
}

\bib{BO:fil}{article}{
author={cgb}, 
author={jso}, 
title={Cube-like structures generated by filters}, 
journal={Algebra Universalis}, 
volume={49}, 
date={2003}, 
pages={129--158}
}

\bib{BO:UniMR}{article}{
title={A Universal Axiomatization of Metropolis-Rota Implication Algebras}, 
author={cgb}, 
author={jso},
status={in preparation} ,
eprint={arXiv:0902.0157v1 [math.CO]}
}

\bib{MR:cubes}{article}{
author={Metropolis, Nicholas}, 
author={Rota,  Gian-Carlo}, 
title={Combinatorial Structure of the faces 
of the n-Cube}, 
journal={SIAM J.Appl.Math.}, 
volume={35}, 
date={1978}, 
pages={689--694}
}


\end{biblist}
\end{bibdiv}
\end{document}